\documentclass[a4,usegraphicx,10pt]{article}

 \usepackage{graphicx}
 \usepackage{tabularx}
 \usepackage{vmargin}
 \usepackage{mathtools}
 \usepackage{subcaption}
 \usepackage{lmodern}
 \usepackage{amsmath,amsthm}
\usepackage{pgfplots,pgfplotstable}
\usepackage{tikz-cd}
\pgfplotsset{compat=1.18}
 \usepackage{amssymb}
 \usepackage{array}
 \usepackage{multirow}
 \usepackage{longtable}
 \usepackage{booktabs}
 \usepackage[english]{babel}
 \usepackage[document]{ragged2e}
 \newtheorem{theorem}{Theorem}[section]
 \newtheorem{lemma}[theorem]{Lemma}
 \newtheorem{definition}[theorem]{Definition}
\newtheorem{algo}[theorem]{Algorithm}
\newtheorem{remark}[theorem]{Remark}

\newcommand{\tn}[1][n]{t^{(#1)}}
\newcommand{\zn}[1][n]{Z^{(#1)}}
\newcommand{\zno}[1][n+1]{Z^{(#1)}}

\newcommand{\ip}[1][u,v]{\left\langle #1 \right\rangle_{L^{2}(\Theta)} }
\newcommand{\ips}{\left\langle}
\newcommand{\ipe}[1][L^{2}(\Theta)]{\right\rangle_{#1}}
\newcommand{\nr}[1][u]{\left\|#1 \right\|_{L^{2}(\Theta_)} }
\newcommand{\nrs}[1][u]{\left\|#1 \right\|^{2}_{L^{2}(\Theta)} }
\newcommand{\nrsw}{\left\|\lw \right\|^{2}_{\mc{K}} }
\newcommand{\nrsf}{\left\|f(\zeta(\pdu)) \right\|^{2}_{\mathcal{L}(\mc{K},L^{2}(\Theta))}}
\newcommand{\nrf}{\left\|f(\zeta(\pdu)) \right\|_{\mathcal{L}(\mc{K}(\Theta),L^{2}(\Theta))}}
\newcommand{\lkl}{\mathcal{L}(\mc{K}(\Theta),L^{2}(\Theta))}
\newcommand{\dms}{(\mc{D}_{m})_{\minn}}
\newcommand{\dm}{\mc{D}_{m}}
\newcommand{\nrb}{\left\|}
\newcommand{\nre}[1][L^{2}(\Theta)]{\right\|_{#1}}
\newcommand{\nres}[1][L^{2}(\Theta)]{\right\|^{2}_{#1}}

\newcommand{\nrls}[1][\Lambda]{\right\|^{2}_{#1}}

\newcommand{\mb}{\mathbb}
\newcommand{\mc}{\mathcal}
\newcommand{\lb}{\left[}
\newcommand{\rb}{\right]}
\newcommand{\xdo}[1][]{X_{\mc{D}#1,0}}
\newcommand{\xdm}[1][N_{m}+1]{X^{#1}_{\mc{D}_{m}}}
\newcommand{\pd}{\Pi_{\mc{D}}}
\newcommand{\un}[1][n]{u^{(#1)}}
\newcommand{\vn}[1][n]{v^{(#1)}}
\newcommand{\uo}{u^{(0)}}
\newcommand{\uno}{u^{(n+1)}}

\newcommand{\bs}{\boldsymbol}
\newcommand{\pdm}[1][m]{\Pi_{\mc{D}_{#1}}}
\newcommand{\Pd}{P_{\mc{D}}}
\newcommand{\Pdm}{P_{\mc{D}_{m}}}
\newcommand{\dd}{\nabla_{\mc{D}}}
\newcommand{\ddm}{\nabla_{\mc{D}_{m}}}
\newcommand{\ddh}{d^{(n+\frac{1}{2})}_{\mc{D}}}
\newcommand{\ddhm}{d^{(n+\frac{1}{2})}_{\mc{D}_{m}}}
\newcommand{\mdm}[1][n]{M_{\mc{D}_{m}}}
\newcommand{\pdu}[1][n]{\pd\un[#1]}
\newcommand{\pduo}{\pd\uno}
\newcommand{\Rm}{(\mc{R}_{m})_{\minn}}
\newcommand{\rmm}{\mc{R}_{m}}
\newcommand{\gdbuo}{\dd \zeta( \uno) }
\newcommand{\gdbu}{\dd \zeta( \un) }
\newcommand{\gdbl}[1][n]{\dd \zeta( \un[#1]) }

\newcommand{\gdbuz}{\dd \zeta( \un[0]) }
\newcommand{\gdbf}[1][u]{\dd \zeta( #1)}

\newcommand{\dlw}{\hat{d} _{L^{2}_{\text{w}}(\Theta)}}

\newcommand{\fr}{\frac}

\newcommand{\pdbul}[1][n]{\pd \zeta( \un[#1]) }

\newcommand{\bpuo}{\zeta( \pd \uno)  }
\newcommand{\bpu}{\zeta( \pd\un)  }
\newcommand{\bpul}[1][n]{\zeta( \pd \un[#1])  }
\newcommand{\bpd}{\bpuo-\bpu }

\newcommand{\buo}{\zeta\left( \uno\right)}
\newcommand{\dn}[1][\cdot]{|#1|_{*,\mc{D}}}
\newcommand{\dne}[1][r]{\right|^{#1}_{*,\mc{D}}}
\newcommand{\bu}{\zeta( \un)}
\newcommand{\zpo}{\Xi( \pduo)}
\newcommand{\zp}[1][n]{\Xi( \pdu[#1])}
\newcommand{\izf}{\ith \Xi( \pd u) }
\newcommand{\lbd}{\Lambda}

\newcommand{\iz}[1][n]{\ith \zp[#1] dx}
\newcommand{\izd}{\ith\left( \zpo-\zp\right)dx }
\newcommand{\lpdo}{\Lambda }
\newcommand{\lpd}[1][n]{\Lambda }
\newcommand{\fpu}[1][u]{f(\pd \zeta ( #1)) }
\newcommand{\fpb}[1][n]{f(\pdbul[#1] ) }
\newcommand{\lw}[1][n+1]{\Delta^{(#1)}W}
\newcommand{\lwm}[1][n+1]{\Delta^{(#1)}\ol{W}_{m}}
\newcommand{\ith}{\int_{\Theta}}
\newcommand{\ol}{\overline}
\newcommand{\hb}{H^{\beta}(0,T,L^{2}(\Theta))}
\newcommand{\hbt}{H^{\beta}(0,T)}
\newcommand{\linf}{L^{\infty}(0,T; L^{2}(\Theta))}
\newcommand{\lrw}{L^{4}(0,T;L^{2}_{\text{w}}(\Theta))}
\newcommand{\ltw}{L^{2}_{\text{w}}(\Theta)}

\newcommand{\ltl}{L^{2}(0,T;L^{2}(\Theta))}
\newcommand{\ltlft}{L^{4}(0,T;L^{2}(\Theta))}

\newcommand{\lilws}{L^{\infty}(0,T;L^{2}(\Theta))_{\text{w*}}}
\newcommand{\ctl}{C(0,T;L^{2}(\Theta))}
\newcommand{\ltldw}{L^{2}(0,T;L^{2})^{d}_{\text{w}}(\Theta)}
\newcommand{\pdmu}{\pdm\ol{u}_{m}}
\newcommand{\Pdmphi}{\pdm P_{\mc{D}_{m}}\phi_{i}}
\newcommand{\bea}[1][]{\begin{equation}\label{#1}\begin{aligned}}
\newcommand{\eea}{\end{aligned}\end{equation}}
\newcommand{\beas}{\begin{equation*}\begin{aligned}}
\newcommand{\eeas}{\end{aligned}\end{equation*}}
\newcommand{\smk}[1]{\sum_{n=0}^{#1}}
\newcommand{\smn}[1][N-\ell]{\sum_{n=1}^{#1}}
\newcommand{\sml}[1][\ell-1]{\sum_{i=0}^{#1}}
\newcommand{\umu}[1][\mu]{\underline{#1}}
\newcommand{\omu}[1][\mu]{\overline{#1}}
\newcommand{\dt}{\delta t_{\mc{D}}}
\newcommand{\dtm}{\delta t_{\mc{D}_{m}}}
\newcommand{\expc}{\mathbb{E}\lb}
\newcommand{\mon}[1][n]{\max\limits_{1\le #1\le N}}
\newcommand{\mxk}[1][k]{\max\limits_{0\le #1\le N-1}}
\newcommand{\onl}{\mathbbm{1}_{[t^{(n)},t^{(n+\ell)}]}(t)}

\newcommand{\trq}{\text{Tr}\left( \mathcal{Q}\right) }
\newcommand{\slm}[1][\phi\in \mc{A}]{\sup\limits_{#1}}

\newcommand{\nf}{\mathop{\rm nf}}
\newcommand{\mdt}{M_{\mc{D}}(t)}
\newcommand{\md}[1][n]{M^{(#1)}_{\mc{D}}}
\newcommand{\mv}[1][u]{\mathbf{#1}}

\newcommand{\minn}[1][m]{#1\in\mb{N}}
\newcommand{\oM}{\ol{M}}
\newcommand{\ou}{\ol{u}}

 \usepackage[T1]{fontenc}
 \usepackage{url}
 \usepackage{fancyhdr}
 \usepackage{psfrag}
 \usepackage{wasysym}
 \usepackage{fancyhdr}
 \usepackage{vmargin}
 \usepackage{psfrag}
\usepackage{bbm}
\usepackage[style=numeric-comp,maxbibnames=99,bibencoding=utf8,firstinits=true]{biblatex}
\usepackage{hyperref}
\numberwithin{equation}{section}
 \setpapersize{A4}
 \setmarginsrb{20mm}{10mm}{20mm}{25mm}{12pt}{8mm}{10pt}{8mm}

\bibliography{references} 

\usepackage{authblk}
\newcommand{\email}[1]{\href{mailto:#1}{#1}}

\usepackage{marginnote}

\newcommand{\logLogSlopeTriangle}[5]
{
    \pgfplotsextra
    {
        \pgfkeysgetvalue{/pgfplots/xmin}{\xmin}
        \pgfkeysgetvalue{/pgfplots/xmax}{\xmax}
        \pgfkeysgetvalue{/pgfplots/ymin}{\ymin}
        \pgfkeysgetvalue{/pgfplots/ymax}{\ymax}

        \pgfmathsetmacro{\xArel}{#1}
        \pgfmathsetmacro{\yArel}{#3}
        \pgfmathsetmacro{\xBrel}{#1-#2}
        \pgfmathsetmacro{\yBrel}{\yArel}
        \pgfmathsetmacro{\xCrel}{\xArel}

        \pgfmathsetmacro{\lnxB}{\xmin*(1-(#1-#2))+\xmax*(#1-#2)} 
        \pgfmathsetmacro{\lnxA}{\xmin*(1-#1)+\xmax*#1} 
        \pgfmathsetmacro{\lnyA}{\ymin*(1-#3)+\ymax*#3} 
        \pgfmathsetmacro{\lnyC}{\lnyA+#4*(\lnxA-\lnxB)}
        \pgfmathsetmacro{\yCrel}{\lnyC-\ymin)/(\ymax-\ymin)}

        \coordinate (A) at (rel axis cs:\xArel,\yArel);
        \coordinate (B) at (rel axis cs:\xBrel,\yBrel);
        \coordinate (C) at (rel axis cs:\xCrel,\yCrel);

        \draw[#5]   (A)-- node[pos=0.5,anchor=north] {\scriptsize{1}}
                    (B)-- 
                    (C)-- node[pos=0.,anchor=west] {\scriptsize{#4}} 
                    cycle;
    }
}

\begin{document}

\title{Numerical analysis of the stochastic Stefan problem}
\author[1]{J\'er\^ome Droniou\footnote{\email{jerome.droniou@monash.edu}}}
\affil[1]{School of Mathematics, Monash University, Australia}
\author[1]{Muhammad Awais Khan\footnote{\email{muhammad.khan@monash.edu}}}
\author[1]{Kim-Ngan Le\footnote{\email{Ngan.Le@monash.edu}}}

\maketitle
 \justify

\begin{abstract}
	The gradient discretisation method (GDM) -- a generic framework encompassing many numerical methods -- is studied for a general stochastic Stefan problem with multiplicative noise. The convergence of the numerical solutions is proved by compactness method using discrete functional analysis tools, Skorokhod theorem and the martingale representation theorem. The generic convergence results established in the GDM framework are applicable to a range of different numerical methods, including for example mass-lumped finite elements, but also some finite volume methods, mimetic methods, lowest-order virtual element methods, etc. Theoretical results are complemented by numerical tests based on two methods that fit in GDM framework.
\end{abstract}

\small\textsc{Keywords}: Stefan equation, stochastic PDE, numerical methods, gradient discretisation method, convergence analysis

\section{Introduction}
In this paper, we study the stochastic Stefan problem (SSP) with multiplicative noise of the form
\bea\label{eq:mp}
&du-\text{div}[\Lambda\nabla\zeta(u)]dt=f(\zeta(u))dW_{t}\quad \text{in}\quad \Theta_{T}:=(0,T)\times \Theta,\nonumber\\&
u(0,\cdot)=u_{0} \quad \text{in}\quad \Theta,\nonumber\\&
\zeta(u)=0\quad \text{on}\quad (0,T)\times \partial\Theta,
\eea
where $T>0$, $\Theta$ is an open bounded domain in $\mb{R}^{d}\left( d\ge 1\right)$, $\Lambda$ is a diffusion coefficient and $W$ is a $\mc{Q}$-Wiener process. Here $\zeta$ is a globally non-decreasing Lipschitz continuous function passing through origin and coercive, see Section \ref{sec:Assum-Nota} for detail.

The deterministic Stefan problem describes the behavior of phase transition during the evolution of two thermodynamical states, say solid and liquid.  This problem has its name from Jo\v{z}ef Stefan (1835-1893) due to his research on solid-liquid phase changes in formation of ice in polar seas \cite{stefan1891theorie}. Deterministic moving boundary problems have been studied extensively in the second half of the past century, see \cite{lunardi2004introduction} and references therein. Numerical analysis of the problem is developed in~\cite{Meyer,Jerome,Elliott} and a huge bibliography can be found in~\cite{Verdi}.  Convergence of variety of gradient schemes (GS) to approximate its solution is provided in \cite{eymard2013gradient} and \cite[Chapter 6]{droniou2018gradient}.

In the past decade, several authors have started investigating stochastic perturbation of the classical Stefan problem. There are numerous examples of applications, such as climate models, diffusion of lithium-ions in lithium-ion batteries, modelling steam chambers for petroleum extraction and oxygen diffusion in an absorbing tissue, which are described by stochastic Stefan-type problems.
  
A significant literature is available on proving the existence and uniqueness of SSP. In \cite{barbu2002two}, a semigroup approach was employed to obtain a mild solution. Using the enthalpy function with additive noise, the two phase free boundary problem is transformed into a stochastic evolution equation of porous media type with the fixed boundary conditions. Applying a coordinate transform, \cite{kim2012stochastic} studied the linear SSP on one dimensional unbounded domains ($d=1$) assuming the random field Brownian in time but correlated (colored) in space. Extending these results,  \cite{keller2016stefan} established the existence and uniqueness of local strong solutions by using the interpolation theory. The analysis is based on the theory of mild and strong solutions proposed in \cite{Prato2014stochastic}. A further extension of these results under Robin boundary conditions was done in \cite{muller2018stochastic} (in which the moving interface between the two phases might have unbounded variation). 

Several works on the numerical approximation of stochastic PDEs driven by Wiener processes have been carried out, such as the heat equation, the Navier--Stokes model, the $p$-Laplace equation, etc. (see, e.g., \cite{droniou2020design,brzezniak2013finite,Hausenblas,gyongy1999lattice,millet2005implicit,walsh2005finite,Yan,zhang2017numerical,diening.hofmanova.ea:23}). However, to the authors' knowledge, there is no available literature on the numerical analysis of the Stochastic Stefan problem.
  
  This paper is primarily focused on the numerical analysis of SSP encompassing well known discrete schemes. We use a generic framework known as gradient discretisation method (GDM) to propose a numerical scheme approximating solutions of SSP. The GDM covers a large class of conforming and nonconforming schemes such as finite volume methods, Galerkin methods (including mass-lumped finite element), mixed finite element methods, hybrid high-order and virtual element methods, etc. A thorough discussion on the analysis and applications of the GDM method is given in the monograph \cite{droniou2018gradient}.  Moreover, we prove the convergence and stability of the numerical solutions, therefore, there is no need to prove convergence for each specific method which simplifies the analysis.
    
Our convergence analysis technique is based on the discrete functional analysis tools, Skorokhod theorem, Kolmogorov test, and martingale representation theorem. We used the similar idea as prescribed in \cite{droniou2020design} to show the existence of a weak martingale solution to the SSP. 

This paper is organised as follows: Section \ref{sec:GS_and_main_result} starts with the introduction of assumptions and notations followed by the description of the GDM framework and the related gradient scheme for the SSP; two examples of methods fitting the framework are presented (the mass-lumped $\mb{P}_{1}$ finite element (MLP1) method and hybrid mimetic method (HMM)), and the section concludes with the statement of the main convergence result. In Section \ref{sec:aprior.estimates} we provide \textit{a priori} estimates of the approximated solutions. Section \ref{sec:compactness} details the following convergence steps: (1) tightness, in appropriate spaces, of  a family of random variables representing in particular the solution to the GS and their gradient, (2) almost sure convergence (up to a change of probability space) in appropriate norms, and (3) continuity of the limits of the numerical solution and of the martingale involving the Wiener process. In Section \ref{sec:identify_limit}, the limit of the martingale part is identified, after which the main convergence result is proved. Numerical results, based on the two examples of GDs presented in Section \ref{sec:GS_and_main_result}, are provided in Section \ref{sec:numerical_examples} to validate the convergence results and their bounds.

\section{Gradient scheme and main results}\label{sec:GS_and_main_result}

\subsection{Assumptions and notations}\label{sec:Assum-Nota}

We first introduce the notations and assumptions that are used in the rest of the paper. For a given $s\in [0,T]$, we set $\Theta_{s}=(0,s)\times \Theta$. Throughout this paper, we assume the following.
\begin{itemize}
	\item[A1.] $\zeta: \mb{R}\rightarrow\mb{R}$ is Lipschitz continuous with Lipschitz constant $L^{-1}_{\zeta}$. We also assume that $\zeta$ is  non decreasing, satisfies $\zeta(0)=0$ and is coercive in the sense that there exists $c,d>0$ such that $|\zeta(s)|\ge c|s|-d$ for all $s\in\mb{R}$.
\end{itemize}
We set  $$\Xi(z):=\int_{0}^{z}\zeta (s)ds, \quad \forall z\in \mb{R}.$$
\begin{itemize}
	\item[A2.]  $\Lambda$ is a symmetric measurable tensor on $\Theta$ that is uniformly elliptic and bounded, that is:
	\beas
	\exists\; \umu, \omu \in (0,\infty)\;\; \text{s.t.}\;\;            \umu|\xi|^{2}\le\Lambda(x) \xi \cdot \xi \le \omu|\xi|^{2}\quad \forall   \xi\in \mb{R}^{d},\;\; \text{for a.e} \;\; x\in \Theta.
	\eeas
 For $F:\Theta\rightarrow \mb{R}^d$, we denote $\nrb F\nrls:=\ips \Lambda F,F\ipe=\int_\Theta \Lambda(x)F(x)\cdot F(x)dx$.
    \item[A3.] $u_{0}:\Theta\rightarrow \mb{R}$ is a measurable function and $\Xi(u_0)\in L^1(\Theta)$.
    \item[A4.] Let   $(\Omega, \mathcal{F},\mathbb F=(\mathcal{ F}_{t})_{t\in [0,T]},\mathbb P)$ be a stochastic basis, that is, $(\Omega, \mathcal{F},\mathbb P)$ is a probability space equipped with an increasing family of $\sigma$-algebras $\{\mathcal{F}_{t}\},\; t\in [0,T]$, called filtration. Assume an $\mathbb F$-adapted Wiener process $W=\{W(t);t\in [0,T]\}$ taking values in a separable Hilbert space $\mathcal{K}(\Theta)$ with covariance operator $\mathcal{Q}$ such that $Tr(\mathcal{Q})<\infty$. $W$ can be written in the form $W(t)=\sum_{k=1}^{\infty}q_{k}W_{k}(t)e_{k},$ where $\{e_{k},k\ge 1\}$ is an orthonormal basis of $\mathcal{K}(\Theta)$ with corresponding eigenvalues $q_{k}$ such that $\sum_{k=1}^{\infty} q^{2}_{k}<\infty$, and  $\{W_{k}:k\ge 1\}$ is a family of independent $\mathbb F$-adapted real-valued Wiener processes.
	\item[A5.] Let $\lkl$ be the Banach space of Hilbert-Schmidt operators \cite[Appendix C]{Prato2014stochastic} with norm denoted as $\nrb \cdot \nre[\lkl]$. We assume that the operator $f:L^{2}\rightarrow \lkl$ is continuous and that there exist $C_{1}, C_{2} >0$ such that for any $w\in L^{2}(\Theta)$ and $\Xi(w)\in L^{1}(\Theta)$
	\bea\label{eq:a1} \left\|f(\zeta (w)) \right\|^{2}_{\mathcal{L}(\mc{K}(\Theta),L^{2}(\Theta))}\le C_{2}\ith\Xi(  w)dx+C_{1}.\eea\\ 
\end{itemize}
Throughout this article, $C$ denotes a generic constant depends on $f,T,\mc{Q},u_{0},\Lambda$ and $\zeta$. We will only mention any further dependencies in its subscript.

	\subsection{Gradient scheme}
	
	We recall here the main definitions of the gradient discretisation method.	
	\begin{definition}\label{def:GD}
		$\mc{D}=(\xdo,\pd,\dd,\mc{I}_{\mc{D}},(\tn)_{n=0,\cdots,N})$ is a space-time gradient discretisation (GD) for the homogeneous Dirichlet boundary conditions, with piecewise constant reconstruction, if the following properties hold:
		\begin{enumerate}
			\item[(i)] $\xdo$ is a finite dimensional vector space of discrete unknowns,
			\item[(ii)] The linear map $\pd:\xdo \rightarrow L^{\infty}(\Theta)$ is a piecewise constant reconstruction operator in the following sense: there exist a basis $(\mv[e]_{i})_{i\in B}$ of $\xdo$ and a family of disjoint subsets $(\Theta_{i})_{i\in B}$ of $\Theta$ such that $\pd u=\sum_{i\in B}u_{i}\mathbbm{1}_{\Theta_{i}}$ for all $u=\sum_{i\in B}u_{i}\mv[e]_{i}\in\xdo$, where $\mathbbm{1}_{\Theta_{i}}$ is the characteristic function of $\Theta_{i}$,
			\item[(iii)] The linear map $\dd: \xdo\rightarrow (L^{2}(\Theta))^{d}$ gives a reconstructed discrete gradient such that the mapping $ v \mapsto \nrb\dd v\nre[L^{2}(\Theta)]$ is a norm on $\xdo$,
			\item[(iv)] 
            $\mc{I}_{\mc{D}}:Y\rightarrow\xdo$ is an interpolation operator that is used to create, from an initial condition in $Y:=\{v\in L^{2}(\Theta): \zeta(v)\in H^1_0(\Theta)\}$, a discrete vector in the space of unknowns,
			\item[(v)] $\tn[0]=0<\tn[1]<\cdots<\tn[N]=T$ is a uniform time discretisation with a constant time step $\dt:=\tn[n+1]-\tn$.
		\end{enumerate}
  
    For any family $(\vn)_{n=0,\cdots,N}\in\xdo^{N+1}$, we defined piecewise-constant-in-time functions $\pd v:[0,T]\rightarrow L^{\infty}(\Theta)$, $\dd v:(0,T]\rightarrow (L^{2}(\Theta))^{d}$ and $\ddh v\in L^\infty$ by: for $n=0,\cdots,N-1$, for any $t\in(\tn,\tn[n+1]]$ and for a.e. $\bs{x}\in \Theta$
	\[
	\begin{aligned}
	&\pd v(0,\bs{x}):=\pd \vn[0](\bs{x}), \quad \pd v(t,\bs{x}):=\pd \vn[n+1](\bs{x}),\\
	&\dd v(t,\bs{x}):=\dd \vn[n+1](\bs{x}),\quad \ddh v(\bs{x}):=\pd \vn[n+1](\bs{x})-\pd \vn(\bs{x}).
	\end{aligned}
	\]
 
     Let $g:\mb{R}\rightarrow\mb{R}$ be such that $g(0)=0$. For $v=\sum_{i\in B}v_{i}\mv[e]_{i}\in\xdo$, we set $g(v)=\sum_{i\in B}g(v_i)\mv[e]_{i}\in \xdo$. The piecewise constant feature of $\pd$ implies \bea\label{eq:pcro}\pd g(v)=g(\pd v)\quad \forall v\in \xdo.\eea
	\end{definition}
    We define the filtration $(\mc{F}^{n}_{N})_{0\le n\le N}$ by:
	 $$\mc{F}^{n}_{N}:=\sigma\{W(t^{k}):0\le k\le n\}\qquad\forall n=0,\ldots,N.$$
 
   \begin{algo}[Gradient Scheme (GS) for \eqref{eq:mp}]\label{algo}
	Set $\un[0]:=\mc{I}_{\mc{D}}u_{0}$ and take random variables $u(\cdot)=(\un(\omega,\cdot))_{n=0,\cdots,N}\in \xdo^{N+1}$ such that $u$ is adapted to the filtration $(\mc{F}^{n}_{N})_{0\le n\le N}$ and, for any $\phi\in\xdo$ and for almost every $\omega\in\Omega$, we have, for $n=0,\cdots, N-1$,
	\begin{equation}\label{eq:gs}
	\ip[\ddh u(\omega), \pd\phi]+\dt\ip[\lpdo\gdbuo(\omega),\dd\phi]=\ip[\fpb(\omega)\lw(\omega),\pd\phi],
	\end{equation}
	where $\lw=W(\tn[n+1])-W(\tn)$.	
\end{algo}

The convergence of the scheme is analysed along a sequence $(\mc{D}_{m})_{\minn}$ of GDs (such a sequence plays the role, in standard mesh-based schemes, of sequences of meshes with size going to zero). To establish this convergence, the sequence must satisfy a few key properties: consistency, limit-conformity and compactness. Regarding the boundedness of $\dtm\|\ddm\zeta(\mc{I}_{\dm}(\phi))\|_\Lambda$ in the definition of consistency, we refer to Remark \ref{rem:consistency}.

    \begin{definition}[Consistency]\label{def:consistency}
 	A sequence $\dms$ of space-time GDs is said to be consistent if
    \begin{enumerate}
        \item[(a)] for all $\phi\in H_{0}^{1}(\Theta)$, we have $S_{\dm}(\phi)\rightarrow 0$ as $m\rightarrow \infty$, where
        \begin{equation}\label{eq:def.SD}
 	  S_{\dm}(\phi):=\min_{w\in\xdm[]}\left(\nrb \pdm w-\phi\nre+\nrb \ddm w-\nabla\phi\nre \right),
        \end{equation}
        \item[(b)]\label{def:consistency.b} for all $\phi\in L^{2}(\Theta)$, $\pdm\mc{I}_{\dm}\phi\rightarrow \phi$ in $L^{2}(\Theta)$ and %
        $\dtm\|\ddm\zeta(\mc{I}_{\dm}(\phi))\|_\Lambda$ remains bounded as $m\rightarrow \infty$,
        \item[(c)] $\dtm \rightarrow 0$ as $m\rightarrow \infty$.
    \end{enumerate} 
\end{definition}
 
 \begin{definition}[Limit Conformity]\label{def:limitconformity}
 	A sequence $\dms$ of space-time GDs is said to be limit conforming if, for all $\bs{\phi}\in H^{\text{div}}(\Theta):=\{\bs{\phi}\in (L^{2})^{d}:\text{div}\bs{\phi}\in L^{2}(\Theta)\}$, we have $W_{\dm}(\phi)\rightarrow 0$ as $m\rightarrow \infty$, where
 	\[
 	W_{\dm}(\bs{\phi}):=\max_{v\in X_{\dm}\setminus\{0\}}\frac{\left|\ip[\ddm v, \bs{\phi}]+\ip[\pdm v,\text{div} \bs{\phi}]  \right| }{\nrb\ddm v\nre}.
 	\]
 \end{definition}
 
\begin{definition}[Compactness]\label{def:compactness} A sequence of $\dms$ of space-time GDs is said to be compact if 
\[
\lim\limits_{\xi\rightarrow 0}\sup\limits_{\minn} T_{\dm}(\xi)=0,
\]
where, extending $\pdm v$ by $0$ outside $\Theta$,
\[
T_{\dm}(\xi):=\max_{v\in X_{\dm}\setminus\{0\}}\frac{\nrb\pdm v(\cdot+\xi)-\pdm v\nre[L^{2}(\mb{R}^{d})]}{\nrb\ddm v\nre[L^{2}(\Theta)]},\quad \forall \xi \in \mb{R}^{d}.
\]
\end{definition}	

	A limit conforming or compact sequence of GDs also satisfy the following important property, which states a (uniform) discrete Poincar\'e inequality.

	\begin{lemma}[Coercivity of sequence of GDs]\label{lem:coercivity}
		If a sequence $\dms$ of space-time GDs is compact or limit conforming then it is coercive: there exists a constant $\rho$ such that 
		\begin{equation}\label{eq:poincare}
		  \max_{v\in X_{\dm}\setminus\{0\}}\frac{\nrb\pdm v\nre}{\nrb\ddm v\nre}\le \rho\quad \forall \minn.
	  \end{equation}
	\end{lemma}
	
	\subsection{Examples of gradient discretisations}\label{sec:examplesGS}
	
	We present here two examples of GDs: the mass-lumped $\mathcal P^1$ (MLP1) finite elements, and the hybrid mimetic mixed method (HMM). Both of them are known to satisfy the properties above under standard mesh regularity assumptions \cite{droniou2018gradient}.
 
 We first specify the Gradient discretisation \ref{def:GD} for MLP1. This GD is based on a conforming triangulation of the domain $\Omega$.
  Let $\xdo=\{v=(v_s)_{s\in \mc{V}}: v_s\in \mb{R},\; v_s=0\;\text{if}\; s\in\partial \Theta \}$, where $\mc{V}$ is the set of all vertices. To define the reconstruction operator $\pd$, we construct a dual mesh $(\Theta_s)_{s\in \mc{V}}$, which can for example be defined by setting $\Theta_s$ as the polygon obtained by linking the centers of the edges and cells having $s$ as a vertex (see \cite[Section 8.4]{droniou2018gradient} for more detail). We then define the piecewise constant reconstruction on this dual mesh $\pd$ by:
 $$\forall u\in\xdo,\quad \forall s\in \mc{V}, \quad \pd u=u_{s}\;\;\text{on}\;\;\Theta_s.$$%
 Further, the reconstructed gradient $\dd$ is the gradient of piecewise linear function based on vertices, i.e., $\dd v=\sum_{s\in V}v_s \nabla \phi_s$ where  $(\phi_s)_{s\in\mc{V}}$ is the canonical basis of the conforming $\mathbb{P}^1$ finite element space on the triangulation.

The second method, HMM, is a polytopal method and therefore can be applied to more general meshes made of polygons, polyhedra, etc. We therefore consider a polytopal mesh (see \cite[Definition 7.2]{droniou2018gradient}). To describe HMM gradient discretisation, let
$$
\xdo=\{v=((v_K)_{K\in \mc{M}},(v_\sigma)_{\sigma\in \mc{E}}): v_K,\;v_\sigma\in \mb{R},\; v_\sigma=0\;\text{if}\; \sigma \subset \partial \Theta\}
$$
be the space of discrete unknowns, where $\mc{M}$ and $\mc{E}$ denotes the sets of cells and edges, respectively. For any $v\in \xdo$, the piecewise constant reconstruction $\pd$ is usually defined by: for all $K\in\mc{M}$ and $v\in\xdo$, $(\pd v)|_{K}=v_K$. However, in the numerical tests of Section \ref{sec:numerical_examples}, we use a slightly modified version of the HMM. Specifically, the reconstruction $\pd$ builds the mass of the solution from both cell and edge unknowns: $\pd v$ is piecewise constant equal to $v_K$ on a domain $\mathfrak D_K$ in $K$, and to $v_\sigma$ on a domain $\mathfrak D_\sigma$ around $\sigma$. In practice, these domains do not need to be specified, only their volume need to be fixed, which is done by selecting a parameter $0\le r\le 1$ and setting $|\mathfrak D_K|=r|K|$, $|\mathfrak D_\sigma|=(1-r)(|K|+|L|)$ (where $K,L$ are the two cells around $\sigma$). The integrals involving the piecewise constant functions thus constructed can then be computed using only these volumes: for example,
\[
\int_\Omega \pd v \pd w=\sum_{K\in\mc{M}}|\mathfrak D_K|v_K w_K+\sum_{\sigma\in\mc{E}}|\mathfrak D_\sigma|v_\sigma w_\sigma.
\]
This modification of the usual HMM ensures that all unknowns (cell and edge) have a strictly positive contribution in the mass matrix, which facilitates the convergence of the Newton iteration (the total number of iterations is reduced). The gradient reconstruct $\dd$ is built from the consistent polytopal gradients
\bea\label{eq:polytopal_gradient}
\nabla_K v=\fr{1}{|K|}\sum_{\sigma\in\mc{E}_{K}}|\sigma|v_\sigma \mv[n]_{K,\sigma}\qquad\forall K\in\mc{M}
\eea
and the stabilizations $R_{K,\sigma}=v_\sigma-v_K-\nabla_k v\cdot (\ol{x}_\sigma-x_K)$, for $K\in\mc{M}$ and $\sigma\in\mc{E}_K$ (set of faces of $K$), where $\mv[n]_{K,\sigma}$ is the outer normal to $K$ on $\sigma$. More precisely, denoting by $D_{K,\sigma}$ the convex hull of the center $\mathbf{x}_K$ of $K$ and $\sigma$, $\dd$ is given by
$$
\begin{aligned}
\forall v\in\xdo\quad \forall K\in \mc{M},\quad \forall \sigma \in \mc{E}_{K}, \quad
 \dd v\big{|}_{D_{K,\sigma}}=\nabla_K v+\fr{\sqrt{d}}{d_{K,\sigma}}R_{K,\sigma}(v)\mv[n]_{K,\sigma},
\end{aligned}
$$
where $d_{K,\sigma}$ is the orthogonal distance between ${x}_K$ and $\sigma\in \mc{E}_{K}$. The stabilization term is required to control the unknowns $(v_K)_{K\in\mc{M}}$ which are not involved in the polytopal gradient.

\subsection{Main results}
We define the solution of \eqref{eq:mp} as follows:
\begin{definition}\label{def:maindefinition}
	Given $T\in (0,\infty)$, a sequence $(\tilde{\Omega},\tilde{\mc{F}}, \tilde{\mb{F}},\tilde{\mb{P}}, \tilde{W}, \tilde{u})
	$ is called a weak martingale solution to \eqref{eq:mp} when it consists of
	\begin{itemize}
		\item[(a)] a probability space  $(\tilde{\Omega},\tilde{\mc{F}}, \tilde{\mb{F}},\tilde{\mb{P}})$ with normal filtration,
		\item[(b)] a $\mc{K}$-valued $\mb{F}$-adapted Wiener process $\tilde{W}$ with the covariance operator $\mc{Q}$,
		\item[(c)] a progressively measurable process $\tilde{u}:[0,T]\times\tilde{\Omega}\rightarrow L^{2}(\Theta)$
	\end{itemize}
such that
\begin{enumerate}
	\item letting $\ltw$ be $L^{2}(\Theta)$ endowed with a weak topology, 
	$\tilde{u}(\cdot,\omega)\in C([0,T];\ltw)$ for $\mb{P}$-a.s.\ $\omega\in \tilde{\Omega}$,
	\item $\expc\sup\limits_{t\in [0,T]}\displaystyle\ith \Xi\left(\tilde{u}(t)\right)\rb<\infty$,
	\item $\zeta(\tilde{u})\in L^2(0,T;H^1_0(\Theta)$) a.s. and $\expc\left\|\zeta\left( \tilde{u}\right)\right\|^{2}_{L^{2}(0,T;H_{0}^{1}(\Theta))}\rb <\infty$,
	\item for every $t\in [0,T]$ and for all $\psi\in H_{0}^{1}(\Theta)$, $\mb{P}$-a.s.\:
	\begin{equation} \label{eq:notion.ws}
		\ip[\tilde{u}(t),\psi]-\ip[u_{0},\psi]-\int_{0}^{t}\ip[\lpdo\nabla \zeta(\tilde{u}(s)),\nabla\psi]ds=\ip[\int_{0}^{t}f(\zeta(\tilde{u}(s,\cdot)))d\tilde{W}(s),\psi],
		\end{equation}
		where the stochastic integral on the left-hand side is the It\^o integral in $L^{2}(\Theta)$. 
\end{enumerate}
\end{definition}

\begin{theorem}\label{th:mt}
	Let $\dms$ be a sequence of GDs that is consistent, limit-conforming, and compact. For every $m\ge1$, there exist a random process $u_{m}$ solution to the GS \eqref{eq:gs} with $\mc{D}:=\dm$.
	Moreover, there exist a weak martingale solution $(\tilde{\Omega},\tilde{\mc{F}}, \tilde{\mb{F}},\tilde{\mb{P}},\tilde{ W},\tilde{ u})$ to \eqref{eq:mp} in the sense of Definition \ref{def:GD} such that for a sequence $(\tilde{u}_m)_m$ of random processes defined on $\tilde{\Omega}$, sharing the same law as $(u_m)_m$, up to a subsequence the following convergences hold: $\mb{\tilde{P}}$-a.s.\ $\omega\in\tilde{\Omega}$,
    $$
    \pdm \tilde{u}_{m}(\omega)\rightarrow\tilde{u}(\omega)\quad \text{weakly-}*\text{in} \quad  L^{\infty}(0,T,L^{2}(\Theta)),
    $$
    $$
    \pdm \zeta(\tilde{u}_{m}(\omega))\rightarrow \zeta(\tilde{u}(\omega)) \quad \text{strongly in} \quad  L^{2}(\Theta_{T}),
    $$
    $$
    \ddm \zeta(\tilde{u}_{m}(\omega))\rightarrow \nabla\zeta(\tilde{u}(\omega)) \quad \text{weakly in} \quad  L^{2}(\Theta_{T})^{d}.
    $$
\end{theorem}

\section{A priori estimates}\label{sec:aprior.estimates}

 In the following lemma, we first provide a priori estimates for the solution $u$ to \eqref{eq:gs} and then deduce its existence. The convergence analysis is carried out along a sequence $(\mc{D}_m)_m$ of GDs that satisfy the consistency, limit-conformity and compactness properties. For ease of notation, in this section we drop the index $m$ and denote by $\mc{D}$ a generic element of that sequence. Moreover, In the proofs, $C$ denote a generic constant which may change from one line to the next but has the same dependency as the constants appearing in the statement of the result to be proved.

\begin{lemma}[Existence and uniqueness for the GS]\label{e&u}
    With the assumptions in \ref{sec:Assum-Nota}, let $\mc{D}$ be a space-time GD as defined in \ref{def:GD}. Then there exists a solution to the Gradient scheme \ref{eq:gs} and, if $u_1$ and $u_2$ are the two solutions of this scheme, then $\zeta(u_1(\omega))=\zeta(u_2(\omega))$ in $\xdo$ and $\pd (u_1(\omega))=\pd (u_2(\omega))$ in $L^\infty(\Theta_T)$ for each $\omega\in\Omega$.
\end{lemma}

\begin{proof}
     For each $\omega$, the set of equations \eqref{eq:gs} for $n=0,\cdots,N-1$ are equivalent to \cite[Eq. (6.10)]{droniou2018gradient} (with $\beta(u)=u$, $\Lambda=\mathrm{Id}$ and $f$ fixed at each time step as $f(\pd\zeta(u^{(n)}(\omega)))\Delta^{(n+1)}W(\omega)\in L^\infty(\Theta)$). Therefore, the existence of a solution $u(\omega)\in \xdo$ and the uniqueness of $\pd u(\omega)\in L^\infty(\Theta)$ and $\zeta(u(\omega))\in\xdo$ is guaranteed by \cite[Corollary 6.9 and Lemma 6.10, respectively]{droniou2018gradient}.
   The lack of complete uniqueness of $u$ means that we have to apply a specific process to ensure we select a (progressively) measurable solution. This can be done the following way. Recall the basis $(\mv[e]_i)_{i\in B}$ from Definition \ref{def:GD}, set $B_0:=\{i\in B\,:\,\text{$\Theta_i$ has zero measure}\}$, and take any solution $w(\omega)$ to \eqref{eq:gs}. We construct a measurable solution time step by time step (and by implicit induction). For $n=0,\ldots,N-1$, write $w^{n+1}(\omega)=\sum_{i\in B}w_i^{n+1}(\omega)\mv[e]_i$ and define then $u^{n+1}(\omega)=\sum_{i\in B}u^{n+1}_i(\omega)\mv[e]_i$ by setting
    \[
    u^{n+1}_i(\omega)=\left\{\begin{array}{ll}
    		w^{n+1}_i(\omega)&\text{ if $i\in B\backslash B_0$},\\
    		\text{inf}\{s:\zeta(s)=\zeta(w^{n+1}_i(\omega))\}&\text{ if $i\in B_0$}.
    		\end{array}\right.
    \]
    This choice ensures that $\pd u^{n+1}(\omega)=\pd w^{n+1}(\omega)$ and that $\zeta(u^{n+1}(\omega))=\zeta(w^{n+1}(\omega))$, and thus that $u^{n+1}(\omega)$ satisfies \eqref{eq:gs}.
	If $i\in B\backslash B_0$, then $\Theta_i$ has a non-zero measure and thus $u^{n+1}_i(\omega)=w^{n+1}_i(\omega)=(\pd w^{n+1}(\omega))_{|\Theta_i}$ is entirely determined by $\pd w^{n+1}$ which, by the uniqueness result of \cite[Lemma 6.10]{droniou2018gradient}, is uniquely determined by \eqref{eq:gs}; since all data in this equation is $\mc{F}^{n+1}_N$-measurable, then so is $\pd w^{n+1}$ and thus $u^{n+1}_i$.
	If $i\in B_0$, then  $u^{n+1}_i(\omega)$ is selected as the smallest value of the (potential) plateau of $\zeta$ at height $\zeta(w^{n+1}_i(\omega))$. Since $\zeta(w^{n+1}(\omega))$ is uniquely determined by \eqref{eq:gs} (see again \cite[Lemma 6.10]{droniou2018gradient}), we infer that $\zeta(w^{n+1}_i)$ is $\mc{F}^{n+1}_N$-measurable, and that the selection $u^{n+1}_i$ of its smallest value on the plateau is also $\mc{F}^{n+1}_N$-measurable.
    Hence, all components of $u^{n+1}$ are $\mc{F}^{n+1}_N$-measurable, which proves the measurability of $u$.%
\end{proof}
\begin{remark}
    Note that, in the examples in the Section \ref{sec:examplesGS}, $B_0$ is empty and thus, in these cases, the uniqueness of $\pd u$ implies the uniqueness of the solution $u$ itself.
\end{remark}
\begin{lemma}[A priori estimates]\label{eq:lemma1}
	There exists a constant $C>0$ such that
	\bea\label{lem1st1}
	\expc\mon\ith \Xi\left( \pd\un[n]\right)+\left\|\dd \zeta\left( u\right)\right\|^{2}_{L^{2}(\Theta_{T})}+\smk{N-1}\left\|\bpuo-\bpu\right\|^{2}_{L^{2}(\Theta)}\rb\le  C.
	\eea
\end{lemma}

    \begin{proof}  

\textbf{Step 1}: \emph{bound on $\zeta(\pd u)$ and $\dd\zeta(u)$}.
	  
		For almost every $\omega\in\Omega$, we choose the test function $\phi=\pdbul[n+1](\omega)$ in \eqref{eq:gs} and write for this $\omega$
		\begin{equation}\label{eq:to2}
		\begin{aligned}
		\ips\ddh u(\omega),\pd \zeta(\un[n](\omega))\ipe&+\dt\nrb\dd \zeta(\un[n+1](\omega)\nrls=\ips f(\pd \zeta(\un[n](\omega)))\lw,\pd \zeta(\un[n+1](\omega))\ipe.
			\end{aligned}
		\end{equation} 
		In rest of the proof, we omit the notation $\omega$. Since $\zeta$ is increasing, $\Xi$ is convex and we have $\Xi(b)-\Xi(a)\le (b-a)\zeta(b)$ for all $a,b\in \mb R$. Substituting $a=\pdu$ and $b=\pduo$, we get
		\begin{equation}\label{eq:betaconvex}
		\ith \left( \zpo-\zp\right)dx\le \ith \ddh u \bpuo dx.
		\end{equation}
		Plugging \eqref{eq:betaconvex} into  \eqref{eq:to2} gives
		\begin{equation}\label{eq:firsttest2}
		\begin{aligned}
		&\izd+\dt \nrb\gdbuo\nrls\\&\le \ips\fpb\lw,\pdbul[n+1]-\pdbul[n]\ipe+\ips\fpb\lw,\pdbul[n]\ipe.
		\end{aligned}
		\end{equation}
		Using another test function $\phi=\buo-\bu$ in \eqref{eq:gs} and the following identity
	  \bea\label{eq:id}
	    \lbd \mv[b]\cdot (\mv[b]-\mv[a])=	\fr{1}{2}\lbd\mv[b]\cdot\mv[b]-\fr{1}{2}\lbd\mv[a]\cdot\mv[a]+\fr{1}{2}\lbd (\mv[b]-\mv[a])\cdot (\mv[b]-\mv[a])
	  \eea
	 for $\mv[a]=\gdbu$ and $\mv[b]=\gdbuo$, we get, from \eqref{eq:gs},
		\begin{equation}\label{eq:secondtest}
		\begin{aligned}
			&\ips\ddh u,\pdbul[n+1]-\pdbul[n]\ipe+\frac{\dt}{2}\left[ \nrb\gdbl[n+1]\nrls-\nrb\gdbl[n]\nrls\rb \\
			& +\frac{\dt}{2}\nrb\gdbuo-\gdbu\nrls\le \ips\fpb\lw,\pdbul[n+1]-\pdbul[n]\ipe.
		\end{aligned}
		\end{equation}
		Moreover, since $\zeta$ is Lipschitz and non decreasing, we have \bea\label{eq:blue}(b-a)(\zeta(b)-\zeta(a))\ge L_{\zeta} |\zeta(b)-\zeta(a)|^{2}\quad \text{for}\quad a,b\in\mb{R}.\eea Substituting $ a=\pdbul$ and $b=\pdbul[n+1]$, we get
		\beas\label{eq:betalip}
		L_{\zeta}\nrb \bpuo-\bpu\nres\le \int_{\Theta}\left(\pduo-\pdu \right)\left( \bpuo-\bpu\right)dx. 
		\eeas
		Plugging that into \eqref{eq:secondtest} and using \eqref{eq:pcro} , we get
		\bea\label{eq:secondtest2}
		 &L_{\zeta}\nrs[\bpuo -\bpu]+\frac{\dt}{2}\left[ \nrb\gdbl[n+1]\nrls-\nrb\gdbl[n]\nrls\rb \\& +\frac{\dt}{2}\nrb\gdbuo-\gdbu\nrls\le \ips\fpb\lw,\pdbul[n+1]-\pdbul[n]\ipe.
		\eea
Adding together \eqref{eq:firsttest2} and \eqref{eq:secondtest2} yields
	\bea\label{eq:Add69}
	&\izd+\dt \nrb\gdbuo\nrls+L_{\zeta}\nrs[\bpuo -\bpu]\\&+\frac{\dt}{2}\left[ \nrb\gdbl[n+1]\nrls-\nrb\gdbl[n]\nrls\rb+\frac{\dt}{2}\nrb\gdbuo-\gdbu\nrls\\&\le 2\ips\fpb\lw,\pdbul[n+1]-\pdbul[n]\ipe+\ips\fpb\lw,\pdbul[n]\ipe.
	\eea
Using a telescopic sum from $n=0$ to $n=k-1$ gives
		\bea\label{eq:t17}
		&\ith \Xi\left( \pd\un[k]\right)+\dt \smk{k-1}\nrb \gdbuo\nrls+L_{\zeta}\smk{k-1}\nrb\bpuo -\bpu\nres+\frac{\dt}{2}  \nrb\gdbl[k]\nrls\\&+\frac{\dt}{2}\smk{k-1}\nrb\gdbuo-\gdbu\nrls\le \ith \Xi\left( \pd\un[0]\right) +\frac{\dt}{2} \nrb\gdbuz\nrls\\&+2\smk{k-1}\ips\fpb\lw, \pdbul[n+1]-\pdbul[n]\ipe+\smk{k-1}\ips\fpb\lw, \pdbul[n]\ipe.
		\eea	
Applying the Cauchy--Schwarz inequality and the Young inequality $ab\le \frac{a}{L_{\zeta}}+L_{\zeta}\frac{b^{2}}{4}$ for the third term of the right-hand side, we obtain
\bea\label{eq:t18}
2\smk{k-1}&\ips\fpb\lw,\pdbul[n+1]-\pdbul[n]\ipe\\&\le 2\smk{k-1}\nr[\fpb\lw]\nrb\pdbul[n+1]-\pdbul[n]\nre\\& \le \frac{2}{L_{\zeta}}\smk{k-1}\nrs[\fpb\lw]+\frac{L_{\zeta}}{2}\smk{k-1}\nrb\pdbul[n+1]-\pdbul[n]\nres.
\eea
Substituting \eqref{eq:t18} into \eqref{eq:t17}, we get
\begin{multline}
  \label{eq:t13}
  \ith \Xi\left( \pd\un[k]\right)+\dt \smk{k-1}\nrb \gdbuo\nrls+\frac{L_{\zeta}}{2}\smk{k-1}\nrb\bpuo -\bpu\nres+\frac{\dt}{2}  \nrb\gdbl[k]\nrls\\
  +\frac{\dt}{2}\smk{k-1}\nrb\gdbuo-\gdbu\nrls\le \ith \Xi\left( \pd\un[0]\right) +\frac{\dt}{2} \nrb\gdbuz\nrls+\\
  \frac{2}{L_{\zeta}}\smk{k-1}\nrs[\fpb\lw]+\smk{k-1}\ips\fpb\lw, \pdbul[n]\ipe.
\end{multline}
We note that the last term of the right-hand side of \eqref{eq:t13} vanishes when taking its expectation since $\pdu$ is $\mathcal F_{\tn}$ measurable and thus independent with $\lw$, which has a zero expectation. Taking expectations, we get
\bea\label{eq:t19}
  &\expc\ith \Xi\left( \pd\un[k]\right)\rb+\dt \expc\smk{k-1}\nrb\gdbuo\nrls\rb+\frac{L_{\zeta}}{2}\expc\smk{k-1}\nrb\bpuo -\bpu\nres\rb\\
  & +\frac{\dt}{2} \expc\nrb\gdbl[k]\nrls\rb+\frac{\dt}{2}\expc\smk{k-1}\nrb\gdbuo-\gdbu\nrls\rb \\
  & \le \ith \Xi\left( \pd\un[0]\right) +\frac{\dt}{2} \nrb\gdbuz\nrls+\frac{2}{L_{\zeta}}\expc\smk{k-1}\nrs[\fpb\lw]\rb.
\eea
From \eqref{eq:a1} and using the independence of the increment of Wiener process, the last term of right-hand side becomes
\bea\label{eq:T144}
\expc\nrs[\fpb\lw]\rb&\le\expc\nrsf\rb\expc\nrsw\rb\\&
\le\dt\text{Tr}(\mathcal{Q})\left( C_{2}\expc\ith\zp\rb+C_{1}\right).
\eea
Together with \eqref{eq:t19}, this implies	
\beas\label{eq:t110}
\expc\ith \Xi\left( \pd\un[k]\right)\rb  \le \ith \Xi\left( \pd\un[0]\right) +\frac{\dt}{2} \nrb \gdbuz\nrls+\frac{2 \text{Tr}(\mathcal{Q})}{L_{\zeta}}\left(C_{2}\smk{k-1}\expc\dt\ith\zp\rb+C_{1}T\right)  .
\eeas
By applying the discrete version of the Gronwall lemma to the above inequality together with the Definition \ref{def:consistency}-(b), we obtain
\bea\label{eq:T112}
\mon\expc\ith\zp\rb\le C.
\eea
It follows from \eqref{eq:t19}-\eqref{eq:T112} that
\beas\label{eq:t113}
&\dt \expc\smk{N-1}\nrb\gdbuo\nrls\rb+\frac{L_{\zeta}}{2}\expc\smk{N-1}\nrs[\bpuo -\bpu]\rb+\frac{\dt}{2} \mon\expc\nrb\gdbu\nrls\rb \\
&+\frac{\dt}{2}\expc\smk{N-1}\nrb\gdbuo-\gdbu\nrls\rb\le  C.
\eeas
\medskip
\textbf{Step 2}: \emph{bound on $\Xi(\pd u)$}.

By taking the maximum of \eqref{eq:t13} over $1\le k\le N$ and applying the expectation, we get
\bea\label{eq:maxofsecondtest2}
\expc\mon[k]\ith \Xi\left( \pd\un[k]\right)\rb
\le{}& \ith \Xi\left( \pd\un[0]\right) +\frac{\dt}{2} \nrb\gdbuz\nrls\\
&+\frac{2}{L_{\zeta}}\expc\smk{N-1}\nrsf\nrsw\rb\\
&+\expc\mon[k]\smk{k-1}\ips\fpb\lw, \pdbul[n]\ipe\rb.
\eea
To bound the third term in the right-hand side, we treat the sum as the stochastic integral of piecewise constant integrand and use the Burkholder--Davis--Gundy inequality \cite[Theorem 2.4]{brzezniak1997stochastic} with constant $B$
\beas\label{eq:BDG}
&\expc\mon[k]\smk{k-1}\ips\fpb\lw, \pdbul[n]\ipe\rb \\
&\le B\expc\left( \sum_{n=1}^{N}\dt\nrsf\nrb\pdbul[n]\nres\right)^{1/2}\rb\\&
\le B \expc\mon\nrf \left(\sum_{n=1}^{N}\dt\nrb\pdbul[n]\nres\right)^{1/2}\rb\\&
\le \frac{1}{4C_2}\expc\mon\nrsf\rb+ B^2C_2 \sum_{n=1}^{N}\expc\dt\nrb\pdbul[n]\nres\rb\\&
\le \frac{1}{4} \expc\mon\ith\zp\rb+\frac{C_{1}}{4C_2}+B^{2}C_2\rho^2 \sum_{n=0}^{N-1}\expc\dt\nrs[\gdbuo]\rb
\eeas
where we have used the Young inequality in the third inequality, and the bound \eqref{eq:a1} together with the Poincar\'e inequality \eqref{eq:poincare} in the fourth inequality.
We use \eqref{eq:T144} to bound the last term of the right-hand side of \eqref{eq:maxofsecondtest2}:
\bea\label{eq:t116}
\expc\sum_{n=1}^{N}\nrsf\nrsw\rb 
\le \text{Tr}(\mathcal{Q})T C_{2}\mon\expc\ith\zp\rb+\text{Tr}(\mathcal{Q})TC_{1}.
\eea
By using \eqref{eq:T112}-\eqref{eq:t116}, we deduce that
 \beas\label{eq:T117}
 \expc\mon\ith\zp\rb\le C.
 \eeas
 which completes the proof of \textit{priori} estimates \eqref{lem1st1}.
 
 The existence of at least one solution $u$ to \eqref{eq:gs} in the Algorithm \ref{algo} follows from the assumption A1, estimate 
 \eqref{lem1st1} and the same arguments as in \cite[Corollary 6.9]{droniou2018gradient}.
 \end{proof}

\begin{remark}[About the definition of consistency]\label{rem:consistency}
The usual definition of consistency for space-time GDs does not require the boundedness of $\dtm\|\ddm\zeta(\mc{I}_{\dm}(\phi))\|_\Lambda$.
Here, this additional assumption is required because, in the stochastic setting, we need to introduce the second
test function $\phi=\buo-\bu$, which gives rise in the left-hand side of \eqref{eq:t17} to the term $\dtm\|\ddm\zeta(\mc{I}_{\dm}(\phi))\|_\Lambda$,
which needs to remain bounded.
\end{remark}

\begin{remark}
    If $W$ is a real-valued Wiener process, inequality~\eqref{eq:T144} can be obtained by using the following assumption on $f$
    $$\|f(\zeta(u)\|^2_{L^2(\Theta)}
        \le
        C_2\int_\Theta\Xi( u)+C_1.$$
\end{remark}

\begin{remark}\label{rm:forubound} By Assumption A1 we have $0\le\zeta'(s)\le L^{-1}_{\zeta}$ for all $s\in\mb{R}$, so
 \[
  \Xi(s)\ge L_{\zeta}\int_{0}^{s} \zeta(w) \zeta'(w)dw
 \]
 and thus
	\bea\label{eq:bound.zeta.Xi}\left|\zeta(s)\right|^{2}\le 2L^{-1}_{\zeta}\Xi (s).
	\eea
	Moreover, since $\zeta$ is coercive by A1, there exists $K_1,K_2$ such that
	\bea\label{eq:rmk2} 
	  s^{2}\le K_{1} \Xi (s)+K_{2}\quad\forall s\in\mb{R}.
	\eea
\end{remark}

 To obtain suitable compactness properties on the sequence of approximate solutions, we will also need an improved version of \eqref{lem1st1} in which we estimate the higher moments of $\Xi(\pd u)$ and $\dd\zeta(u)$.
 
 \begin{lemma}\label{eq:lemma2}
 	Let $u$ be a solution to \eqref{eq:gs}. Then there exists a constant $C>0$ such that
 	\bea\label{lem1st2}
 	\expc\mon \left(\ith \Xi\left( \pd\un[n] \right)\right)^{2}+\left\|\dd \zeta\left( u\right)\right\|^{4}_{L^{2}(\Theta_{T})} \rb\le  C.
 	\eea
 	\end{lemma}

  \begin{remark}
  We could also, following the technique in \cite{droniou2020design}, bound the moments of order $q$ for any $q\ge 1$, instead of $q=2$ as above, but these bounds will not be useful to us here.
  \end{remark}
  
\begin{proof}
Let 
  $$B_{k}:=\sum_{i=0}^{k}\dt \nrb \gdbl[i]\nrls$$
and 
  $$Z^{(n)}:= \iz+\frac{\dt}{2} \nrb\gdbl[n]\nrls+B_{n}.$$

\textbf{Step 1}: \emph{bound on $\max_n \expc ( \zn)^{2}\rb$}.
 
We use the above notations to rewrite \eqref{eq:Add69} as
\beas
&\zno-\zn+L_{\zeta}\nrs[\bpuo -\bpu]+\frac{\dt}{2}\nrb\gdbuo-\gdbu\nrls\\&\le 2\ips\fpb\lw,\pdbul[n+1]-\pdbul[n]\ipe+\ips\fpb\lw,\pdbul[n]\ipe.
\eeas
Multiplying the above equation with $\zno$, we have
 \bea\label{eq:t123}
  \zno{}&\left( \zno-\zn\right)+L_{\zeta}\zno\nrs[\bpuo -\bpu]
  \\
  \le{}& 2\zno\ips\fpb\lw,\pdbul[n+1]-\pdbul[n]\ipe\\&+\zno\ips\fpb\lw,\pdbul[n]\ipe=I_{1}+I_{2}.
 \eea
 
 We use the Cauchy--Schwarz and Young inequalities to estimate $I_{1}$:
\beas
 I_{1}  
  \le{}& 2\zno\nrb\fpb\nre[\lkl]\nrb\lw\nre \nrb\bpd\nre \\
  \le{}& \fr{2}{L_{\zeta}}\zno\nrb\fpb\nre[\lkl]^{2}\nrb\lw\nre^{2}\\
       & +\fr{L_{\zeta}}{2}\zno \nrb\bpd\nres \\
 	\le{}& \fr{2}{L_{\zeta}}\left( \zno-\zn\right) \nrb\fpb\nre[\lkl]^{2}\nrb\lw\nre^{2}\\
 	     & +\fr{2}{L_{\zeta}}\zn \nrb\fpb\nre[\lkl]^{2}\nrb\lw\nre^{2}\\&+\fr{L_{\zeta}}{2}\zno \nrb\bpd\nres \\
 	\le{}& \fr{1}{8}\left( \zno-\zn\right)^{2} +\fr{8}{L_{\zeta}^{2}}\nrb\fpb\nre[\lkl]^{4}\nrb\lw\nre^{4}\\
 	     & +\fr{2}{L_{\zeta}}\zn \nrb\fpb\nre[\lkl]^{2}\nrb\lw\nre^{2}\\&+\fr{L_{\zeta}}{2}\zno \nrb\bpd\nres.
 \eeas
 To estimate $I_2$, we proceed with
 \beas
 I_{2}={}& \zno\ips\fpb\lw,\pdbul[n]\ipe\\
      ={}&\left( \zno-\zn\right)\ips\fpb\lw,\pdbul[n]\ipe\\
        &   + \zn\ips\fpb\lw,\pdbul[n]\ipe\\
      \le{}&\fr{1}{16}\left(\zno-\zn\right)^{2}+4\nrb\fpb\nre[\lkl]^{2}\nrb\lw\nres \nrb\pdbul[n]\nres \\
          & +\zn\ips\fpb\lw,\pdbul[n]\ipe.
 \eeas
 We note that, by \eqref{eq:a1} and \eqref{eq:bound.zeta.Xi},
	\begin{equation}\label{eq:bound.f.z}
		\begin{aligned}
			\nrb\fpb\nre[\lkl]^2\le{}& C_2 \zn+C_1,\\
			\nrb\pdbul[n]\nres\le{}& 2L^{-1}_{\zeta}\zn.
		\end{aligned}
	\end{equation}
Using the estimates on $I_1$ and $I_2$ together with \eqref{eq:id},  \eqref{eq:t123} and \eqref{eq:bound.f.z}, we get
\bea\label{eq:t125}
 &\fr{1}{2}\left( \zno\right) ^{2}-\fr{1}{2}\left( \zn\right) ^{2}+\fr{5}{16}\left(\zno-\zn \right)^{2}
 +\fr{L_{\zeta}}{2}\zno\nrs[\bpuo -\bpu] \\
 &\le \fr{8}{L_{\zeta}^{2}}(C_2 \zn+C_1)^{2}\nrb\lw\nre^{4}
 +\fr{10}{L_{\zeta}}\zn (C_2 \zn+C_1)\nrb\lw\nre^{2}\\
 &+\zn\ips\fpb\lw,\pdbul[n]\ipe.
 \eea
 We note that the last term on the right-hand side of \eqref{eq:t125} vanishes when taking expectation while the first two terms are estimated as follows:
	\bea\label{eq:t132}
	& \fr{8}{L_{\zeta}^{2}}\expc(C_2 \zn+C_1)^{2}\nrb\lw\nre^{4}\rb
	+\fr{10}{L_{\zeta}}\expc\zn (C_2 \zn+C_1)\nrb\lw\nre^{2}\rb\\
	&\le
	\fr{8}{L_{\zeta}^{2}}\dt^{2} \left( \trq\right)^{2}\expc (C_2 \zn+C_1)^{2}\rb+\fr{10}{L_{\zeta}}\dt \trq\expc \zn (C_2 \zn+C_1)\rb.
	\eea
  Hence, summing \eqref{eq:t125} from $n=0$ to $n=k$ where $k=0,\cdots, N-1$, taking expectation and using \eqref{lem1st1}, the above estimates and the discrete Gronwall lemma, we obtain
    \bea\label{eq:t135}
    \mon \expc \left( \zn\right)^{2}\rb\le C. 
    \eea
    
    \textbf{Step 2}: \emph{conclusion of the proof}.

    Summing \eqref{eq:t125} from $n=0$ to $n=k$ where $k=0,\cdots, N-1$,
   
     \beas\label{eq:t136}
    \fr{1}{2}\left( \zn[k+1]\right) ^{2}\le{}& \fr{1}{2}\left( \zn[0]\right) ^{2}+\fr{8}{L_{\zeta}^{2}}\smk{k}\nrb\fpb\nre[\lkl]^{4}\nrb\lw\nre^{4}\\&+
    \fr{2}{L_{\zeta}}  \smk{k}\nrb\fpb\nre[\lkl]^{2}\nrb\lw\nre^{2}\zn \\
    &+4\smk{k}\nrb\fpb\nre[\lkl]^{2}\nrb\lw\nres \nrb\pdbul[n]\nres \\& + \smk{k}\ips\fpb\lw,\pdbul[n]\ipe\zn.
    \eeas
     and taking maximum over $k$ and then applying expectation, we get
       \bea\label{eq:t137}
    \expc\mon\left( \zn\right) ^{2}\rb\le{}& \left( \zn[0]\right) ^{2}+\fr{16}{L_{\zeta}^{2}}\expc\smk{N-1}\nrb\fpb\nre[\lkl]^{4}\nrb\lw\nre^{4}\rb
    \\&+
   \fr{4}{L_{\zeta}}  \expc\smk{N-1}\nrb\fpb\nre[\lkl]^{2}\nrb\lw\nre^{2}\zn\rb\\& +8\expc\smk{N-1}\nrb\fpb\nre[\lkl]^{2}\nrb\lw\nres \nrb\pdbul[n]\nres\rb  \\& +2\expc \mxk\smk{k}\ips\fpb\lw,\pdbul[n]\ipe\zn\rb.
   \eea
Using the Burkholder--Davis--Gundy inequality (with constant $B$) and \eqref{eq:bound.f.z} we have
   \bea\label{eq:t138}
   &2\expc \mxk\smk{k}\ips\fpb\lw,\pdbul[n]\ipe\zn\rb\\ & \le
   2B\expc  \left( \smn[N-1]\dt \nrb \fpb\nres[\lkl] \nrb\pdbul[n]\nres \left( \zn\right)^{2} \right)^{1/2} \rb\\
   &\le
   2B\expc \mxk[n] \zn  \left( \smn[N-1]\dt \left(C_{2}\zn+C_{1} \right)  \times 2L^{-1}_{\zeta}\zn\right)^{1/2} \rb\\
   &\le
   \fr{1}{2}\expc \mxk[n] \left( \zn\right)^{2}\rb+4L^{-1}_{\zeta}B^{2}\smn[N-1]\dt\expc C_2(\zn)^2+C_1\zn\rb
    \\& \le
   \fr{1}{2}\expc \mxk[n] \left( \zn\right)^{2}\rb+4L^{-1}_{\zeta}B^{2}T\left((C_2+C_1)\mxk[n]\expc (\zn)^2\rb+C_1\right),
   \eea
   where the conclusion follows from $\zn\le 1+(\zn)^2$.
   
   The estimates \eqref{eq:t132}--\eqref{eq:t135}, \eqref{eq:t137} and \eqref{eq:t138} yield
   \beas
   \expc \mon \left( \zn\right)^{2}\rb\le C,
  \eeas
  which gives us the required estimate \eqref{lem1st2}.
\end{proof}

 \begin{remark}\label{rm:ubound}
	From \eqref{eq:rmk2} and Lemma \ref{eq:lemma2}, we have the following estimate, for $r\le 4$:
	\bea\label{eq:ubound}
	\expc \nrb \pd u\nre[L^{\infty}(0,T,L^{2}(\Theta))]^{r}\rb \le C.
	\eea
\end{remark}

Obtaining strong compactness result on the sequence of approximate solutions will require to estimate the time translates of $\zeta\left(\pd u\right)$. The following lemma is the key element for this estimate.

\begin{lemma}\label{lem:2}
	Let $u$ be the solution of \eqref{eq:gs}. Then there exists a constant $C>0$ such that, for all $\ell\in \left\lbrace 1,\cdots ,N-1 \right\rbrace $
	$$\expc \dt \smk{N-\ell}\nrb\bpul[n+\ell]-\bpu\nres\rb\le \tn[\ell]C.$$
\end{lemma}

\begin{proof}
	For any $\phi \in \xdo$, writing \eqref{eq:gs} for a generic $i\in\{0,\ldots,\ell\}$ and summing over $i$ yields
 \bea\label{eq:t221}
	\sml\ips \pdu[n+i+1]-\pdu[n+i],\pd \phi\ipe ={}& -\dt \sml\ips \lpd[n+i+1] \gdbl[n+i+1],\dd \phi\ipe \\
	&+\sml\ips \fpb[n+i]\lw[n+i+1],\pd \phi\ipe.
	\eea
	Choosing $\phi=\dt\left( \bpul[n+\ell]-\bpul[n]\right) $ and taking the sum over $n$ from $1$ to $N-\ell$ gives
	\beas
	\dt\smn {}&\ips \pdu[n+\ell]-\pdu[n],\left( \bpul[n+\ell]-\bpul[n]\right) \ipe\\
	  ={}& -\dt^{2} \smn \sml\ips \lpd[n+i+1] \gdbl[n+i+1],\dd \left(  \bpul[n+\ell]-\bpul[n]\right) \ipe \\
	  &+\dt\smn\sml\ips \fpb[n+i]\lw[n+i+1],\left( \bpul[n+\ell]-\bpul[n]\right)\ipe.
	\eeas
	Applying \eqref{eq:blue} for $ a=\pdbul$ and $b=\pdbul[n+1]$, we have
	\bea\label{eq:t223}
	\dt L_{\zeta} \smn{}& \nrb \bpul[n+\ell]-\bpul\nres\\
	& \le -\dt^{2} \smn \sml\ips\lpd[n+i+1]\gdbl[n+i+1],\dd \left(  \bpul[n+\ell]-\bpul[n]\right)\ipe \\
	&+\dt\smn\sml\ips \fpb[n+i]\lw[n+i+1],\left( \bpul[n+\ell]-\bpul[n]\right)\ipe =: I_{1}+I_{2}.
	\eea
	We first estimate the expectation of $I_{1}$ by using Cauchy--Schwarz inequalities
    \begin{align}
    &\expc I_{1}\rb\le
    \omu\expc \dt \smn \nrb\dd \left(  \bpul[n+\ell]-\bpul[n]\right)\nre\sml\dt\nrb\gdbl[n+i+1]\nre\rb\nonumber\\
    &=
	\omu\expc \dt \smn \nrb\dd \left(  \bpul[n+\ell]-\bpul[n]\right)\nre\int_{\tn[n]}^{\tn[n+\ell]}\nrb\dd \zeta \left(u \right) \nre\rb \nonumber\\
	&\le
	\omu\left(\tn[\ell]\right)^{\frac12}\expc \dt \smn \nrb\dd \left(  \bpul[n+\ell]-\bpul[n]\right)\nre\left( \int_{\tn[n]}^{\tn[n+\ell]}\nrb\dd \zeta \left(u \right) \nres\right) ^{\frac{1}{2}}\rb\nonumber\\
	&\le
	\omu\left(\tn[\ell]\right)^{\frac12}\expc \left(\dt  \smn \nrb\dd \left(  \bpul[n+\ell]-\bpul[n]\right)\nres\right) ^{\frac12}\left(\smn \dt\int_{\tn[n]}^{\tn[n+\ell]}\nrb\dd \zeta \left(u \right) \nres\right) ^{\frac{1}{2}}\rb\nonumber\\
	&\le
	\omu\left(\tn[\ell]\right)^{\frac12}\expc \left(2\int_{\tn[\ell]}^{\tn[N]}\nrb\dd \zeta \left( \pd u\right)\nres+2\int_{\tn[0]}^{\tn[N-\ell]}\nrb\dd \zeta \left( \pd u\right)\nres \right) ^{\frac{1}{2}}\left(\ell\dt\int_{\tn[1]}^{\tn[N]}\nrb\dd \zeta \left( \pd u\right)\nres\right) ^{\frac{1}{2}}\rb	\nonumber\\&\le
		2\omu\tn[\ell]\expc \nrb\dd \zeta \left( \pd u\right)\nres[L^{2}\left(\Theta_{T}\right) ]\rb,
	\label{eq:t224}
	\end{align}
 where we have used, in the penultimate line, the fact that, in the term $\smn \int_{\tn[n]}^{\tn[n+\ell]}$, each integral over $[\tn[k],\tn[k+1])$ for $k=n,\ldots,N$, appears at most $\ell$ times.
	To estimate the expectation of $I_{2}$ we use the Young inequality and write
	\begin{align}\label{we.need.a.number.here}
	\expc I_{2} \rb = {}&\dt\expc\smn\ips \int_{0}^{T} \onl \fpu d W\left( t\right) ,\left( \bpul[n+\ell]-\bpul[n]\right)\ipe\rb\nonumber\\
	 \le{}&
	\frac{L_{\zeta}}{4}\dt\smn \expc\nrb \bpul[n+\ell]-\bpul[n]\nres\rb\nonumber\\
	&+\frac{\dt}{L_{\zeta}}\smn\expc \nrb  \int_{0}^{T} \onl \fpu d W(t)\nres \rb.
	\end{align}
	By using It\^{o} isometry, Lemma \ref{eq:lemma1} and \eqref{eq:a1}, we bound the last term of the right-hand side:
	\bea\label{eq:t226}
	\frac{\dt}{L_{\zeta}}\smn{}&\expc \nrb  \int_{0}^{T} \onl \fpu  d W\left( t\right)\nres \rb\\ 
	& \le
	\frac{\trq \dt}{L_{\zeta}}\smn\expc \int_{0}^{T} \onl \nrb \fpu\nres[\lkl]dt\rb \\
	& \le
	\frac{\trq \ell \dt}{L_{\zeta}}\expc \int_{\tn[1]}^{\tn[N]} \nrb \fpu\nres[\lkl]dt\rb \\&\le
	\frac{\trq \tn[\ell] C_{2}}{L_{\zeta}}\left( \expc \int_{0}^{T} \ith\Xi\left( \pd u\right)\rb+TC_{1}\right),
	\eea
	where we have used, in the second equality, the bound $\smn\onl\le \ell$ (which follows, as in \eqref{eq:t224}, from the fact that in this sum each interval $[\tn[k],\tn[k+1]]$ appears at most $\ell$ times.
	Combining \eqref{eq:t223}--\eqref{eq:t226} and Lemma \ref{eq:lemma1} concludes the proof.
\end{proof}

Lemma \ref{lem:2} and \cite[Lemma A.2]{droniou2020design} give the result below
\bea[eq:hb]
\expc\nrb \zeta\left( \pd u\right) \nres[H^{\beta}(0,T,L^{2}(\Theta))]\rb\le C_{\beta}, \qquad \text{for any} \quad \beta \in (0, 1/2).
\eea
In the following lemma, we estimate the dual norm of the time variation of the iterates $\{\pdu\}_{n=0}^{N}$.
The dual norm $\dn$ on $\pd(\xdo)\subset L^{2}(\Theta)$ is defined by the following: for all $v\in \pd (\xdo)$,
	\beas
	\dn[v]:=\sup\left\lbrace\int_{\Theta}v(x)\pd \phi(x) dx\;:\; \phi\in\mc{A}\right\rbrace, 
	\eeas
	where $\mc{A}:=\left\lbrace \phi \in \xdo, \nrb \pd\phi\nre+\nrb\dd \phi\nre\le 1\right\rbrace $. We note that, for all $w\in \xdo$,
	\bea\label{est:adjoint.norm}
	  \left|\int_\Theta \pd w\pd\psi\right|\le \dn[\pd w]\left(\nrb\pd\psi\nre +  \nrb\dd\psi\nre\right)\qquad\forall\psi\in\xdo.
	\eea
 
 \begin{lemma}\label{lam:31}
For all $\ell=1,\cdots,N-1$, there exists a constant $C>0$ such that
	\bea\label{eq:lem3.4}
	\expc \left|\pdu[n+\ell]-\pdu\dne[4]\rb\le C\; \left( \tn[\ell]\right) ^{2}.
	\eea
	As a consequence, for any $t,s\in[0,T]$
	\bea\label{lam:32}
		\expc \left|\pd u(t)-\pd u(s)\dne[4]\rb\le C\; \left( |t-s|+\dt\right) ^{2}.
	\eea
\end{lemma}

\begin{proof}
We apply \eqref{eq:t221} to a generic $\phi\in\mc{A}$ to get
	\bea\label{eq:t4I12}
	\expc\left|\pdu[n+\ell]-\pdu[n]\dne[4]\rb={}&\expc\left|\slm\ips \pdu[n+\ell]-\pdu[n], \pd \phi \ipe\right|^{4}\rb\\
	\le{}& 2^{3}(\dt)^{4} \expc\left( \slm\sml \ips\lbd \gdbl[n+i+1],\dd \phi \ipe\right) ^{4}\rb \\
	&+2^{3}\expc\left(\slm \sml\ips \fpb[n+i]\lw[n+i+1],\pd \phi\ipe\right) ^{4}\rb:=I_{1}+I_{2}.
	\eea
To estimate $I_{1}$, we use lemma \ref{eq:lemma2} and the H\"{o}lder inequality, 
to obtain
    \beas\label{eq:t4I1}
	I_{1}& \le
	C\omu^{4}(\dt)^{4} \expc\left( \sml \nrb \gdbl[n+i+1]\nre \slm\nrb\dd \phi \nre\right)^{4}\rb \\& \le
	C\omu^{4}\expc \left( \int_{\tn[n]}^{\tn[n+\ell]} \nrb \gdbf\nre \right) ^{4}\rb \\& \le
	C\omu^{4}\left( \tn[\ell]\right)^{2} \expc \left(  \int_{\tn[n]}^{\tn[n+\ell]} \nrb \gdbf\nres\right) ^{2}\rb \\& \le
	C\omu^{4}\left( \tn[\ell]\right)^{2} \expc \nrb \gdbf\right\|^{4}_{L^{2}(\Theta_{T})}\rb \le C \left( \tn[\ell]\right)^{2}. 
    \eeas
We estimate $I_{2}$ using the Burkholder--Davis--Gundy inequality, \eqref{eq:a1} and \eqref{lem1st2} to write
\bea\label{eq:t4I2}
I_{2}&\le B\expc\left( \nrb \int_{0}^{T}\onl\fpu dW(t)\nre^{4} \slm\nrb\pd \phi\nre^{4}\right) \rb \\& \le
B\expc\left(  \int_{\tn}^{\tn[n+\ell]}\nrb\fpu\right\|^{2}_{\lkl} dt \right)^{2} \rb \\& \le
B \left( \tn[\ell] \right)^{2} \expc\int_{\tn}^{\tn[n+\ell]}\nrb\fpu\right\|^{4}_{\lkl} dt  \rb \\& \le
B \left( \tn[\ell]\right)^{2}  \expc\int_{\tn}^{\tn[n+\ell]} \left(C_{2}\left( \izf\right)^{2} +C_{1}\right) dt  \rb \\& \le
C \left(\tn[\ell] \right)^{2}
\eea
The estimate \eqref{eq:lem3.4} follows from \eqref{eq:t4I12}-\eqref{eq:t4I2}. The estimate \eqref{lam:32} follows by the same argument as in the proof of \cite[Lemma 3.5]{droniou2020design}.
\end{proof}

Let $\left\lbrace \psi_{i}, \minn[i] \right\rbrace\subseteq C^{\infty}_{c}(\Theta)\setminus\left\lbrace 0 \right\rbrace  $ be a countable dense subset in $L^{2}(\Theta)$ and we set
\bea
 \label{eq:phidef}\phi_{i}:=\psi_{i}/(\nrb \psi_{i}\nre+\nrb \nabla \psi_{i}\nre).
\eea
 We also define the interpolator $P_{\mc{D}}: H^{1}_{0}(\Theta)\cap L^{2}(\Theta)\rightarrow\xdo $ by
\bea\label{eq:pdm}
P_{\mc{D}} \phi := \text{argmin}_{w\in\xdo}\left( \nrb \pd w-\phi\nre +\nrb \dd w -\nabla \phi\nre \right). 
\eea
From \eqref{eq:phidef}, \eqref{eq:pdm} and using the definition of consistency of space-time GDs with $w=0$ together with $\nrb \phi_{i}\nre+\nrb \nabla \phi_{i}\nre=1$, we have
	\begin{align}
	  \nrb \pd \Pd \phi_{i}\nre+\nrb\dd \Pd \phi_{i} \nre\le{}& \nrb \pd \Pd \phi_{i}-\phi_i\nre+\nrb\dd \Pd \phi_{i} -\nabla\phi_i\nre 
	  +\nrb \phi_{i}\nre+\nrb\nabla \phi_{i} \nre \nonumber\\
	  ={}& S_{\mc{D}}(\phi_i)+\nrb \phi_{i}\nre+\nrb\nabla \phi_{i} \nre\le 2.
	\label{eq:pdmpe}
	\end{align}
 
 The following lemma gives a bound on the time variations of $\left( |\ip[\pd u,\pd \Pd \phi_{i}]|\right)_{\minn[i]}$.
\begin{lemma}\label{lem:hbt}
	For $\beta \in (0,\frac{1}{2})$, there exists a constant $C>0$ depending on $\beta$ such that
	\beas
	\expc \nrb \ip[\pd u,\pd \Pd \phi_{i}]\nre[\hbt]^{2}\rb\le C_{\beta}.
	\eeas

\end{lemma}

\begin{proof}
	Using the Cauchy--Schwarz inequality, \eqref{est:adjoint.norm}, \eqref{eq:pdmpe} and Lemma \ref{lam:31}, we obtain the following estimate
	\beas
	&\expc \dt \smn \left| \ips\pdu[n+\ell]-\pdu, \pd \Pd \phi_{i}\ipe\right|^{2}\rb
	\\& \le \expc \dt \smn \left| \ips\pdu[n+\ell]-\pdu, \pd \Pd \phi_{i}\ipe\right|^{4}\rb^{\frac{1}{2}}\left( \dt (N-\ell) \right)^{\fr{1}{2}}
	\\& \le  4\expc\dt \smn  \left|\pdu[n+\ell]-\pdu\dne[4]\rb^{\frac{1}{2}}T^{1/2}
	\\& \le  C\; \tn[\ell].
	\eeas
 Finally, reasoning as in the proof of \cite[Lemma A.2]{droniou2020design}, we have the required estimate.
\end{proof}

For any $t\in [0,T]$, we define
$$
  \mdt =\md:=\sum_{i=0}^{n}\fpb[i]\lw[i+1] \quad \text{for}\quad t\in [\tn, \tn[n+1]], \quad n\in \{0,\cdots,N-1\}.
$$
In the following lemma, we estimate $M_{\mc{D}}$. 

\begin{lemma}\label{t:4}
	For any $\beta\in (0,1/2)$ there exists  a constant $C>0$ depending on $\beta$ such that
	\bea[eq:l41]
	\expc \nrb M_{\mc{D}} \nre[\linf]^{4}\rb\le C \qquad \text{and} \qquad \expc \nrb M_{\mc{D}} \nres[\rm{W}^{\beta,4}(0,T,L^{2}(\Theta))]\rb\le C_{\beta}.
	\eea
\end{lemma}

\begin{proof}
Using the Burkholder--Davis--Gundy inequality, we have
	\beas
	\expc\nrb M_{\mc{D}} \nre[\linf]^{4}\rb= \expc \mon\nrb \sum_{i=0}^{n} \fpb[i]\lw[i+1]\nre^{4}\rb\le \expc\left( \int_{0}^{T} \nrb \fpu\nres[\lkl]\right) ^{2}\rb.
	\eeas
	The first estimate is then follows from \eqref{eq:a1} and Lemma \ref{eq:lemma2}.
 
For the second estimate, it follows from \eqref{eq:t4I2}, that 
	\bea\label{eq:Mtimetrans}
	\expc \nrb\md[n+\ell]-\md[n]\nre^{4}\rb= \expc\nrb \int_{0}^{T} \onl \fpu dW(t)\nre^{4}\rb\le C (\tn[\ell])^{2}.
	\eea
	Together with the first estimate in \eqref{eq:l41} and reasoning as in \cite[Lemma A.2]{droniou2020design}, we obtain the second estimate in \eqref{eq:l41}.
\end{proof}	

\section{Compactness of the solution to the scheme}\label{sec:compactness}

From hereon, we re-introduce the index $m$ in the sequence of gradient discretisations $(\mc{D}_m)_m$, removed at the start of Section \ref{sec:aprior.estimates}, and prove the tightness of the law of the sequence
$$
(\mc{R}_{m})_{\minn}:=\big(\pdm u_{m}, \pdm\zeta\left(u_{m}\right), \ddm \zeta (u_{m}), \mdm, W,\left(\ips \pdm u_{m},\pdm\Pdm \phi_{i}\ipe \right)_{\minn[i]} \big)_{m\in\mb{N}}
$$
in the space $$\mc{E}:= \lilws\times L^{2}(0,T;L^{2})\times L^{2}(0,T;L^{2})^{d}_{\text{w}}\times L^{4}(0,T;L^{2}_{\text{w}})\times C([0,T];L^{2})\times L^{4}(0,T)^{\mb{N}}$$ where the spaces $\lilws$ and $L^{2}(0,T;L^{2}(\Theta))^{d}_{\text{w}}$ are the spaces endowed with the weak-$*$ and weak topologies, respectively. %
For the definition of the tightness, we refer the reader to the Appendix \ref{appen}.%

To prove the tightness of $(\pdm \zeta (u_{m}))_m$, we define the following norm on $\xdm:$ for any $w_{m}\in \xdm$ \bea[eq:dmnorm]\nrb w_{m}\nre[\beta,\mc{D}_{m}]:=\nrb\ddm w_{m}\nre[L^{2}(\Theta_{T})]+\nrb \pdm w_{m}\nre[\hb].\eea
By Lemma \ref{eq:lemma1} and Estimate \eqref{eq:hb}, we have 
\bea[eq:dmzeta]
  \expc \nrb \zeta (u_{m})\nres[\beta,\mc{D}_{m}]\rb\le C_{\beta}.
\eea
Recalling that $\{\phi_i\}_i$ defined by \eqref{eq:phidef} form a family of smooth functions whose span is dense in $L^2(\Theta)$, we define the metric
$$
  d_{L^2_{\text{w}}}(u,v)=\sum_{\minn[i]}\frac{\left|\ips u-v,\phi_i\ipe \right|}{2^{i}}\quad\text{for}\; u,v\in L^2(\Theta).
$$ 
On bounded sets in $L^2(\Theta)$ this metric defines the weak topology. To prove the tightness of $(\mdm)_m$, we introduce the space $L^{r}(0,T;L^{2}_{\text{w}}(\Theta))$ for $2\le r<4$, which is $L^{r}(0,T;L^{2}(\Theta))$ endowed with the metric 
$$
  d_{L^{r}\left( L^2_{\text{w}}\right) }(u,v):=\left(\int_{0}^{T}d_{L^2_{\text{w}}}(u(s),v(s))^{r} \right) ^{\fr{1}{r}}.
$$ 
For all $\phi\in L^{2}(\Theta)$, the map $L^{r}(0,T;\ltw) \ni v\mapsto\left\langle v(\cdot),\phi\right\rangle_{L^2(\Theta)}\in L^{r}(0,T) $ is continuous. Further, if $(v_{n})_{n\in\mb{N}}$ is bounded in $L^{\infty}(0,T;L^{2}(\Theta))$ then $v_{n}$ converges to $v$ in $\lrw$ if and only if for all $\phi \in L^{2}(\Theta)$:
\[
  \left\langle v_{n}(\cdot),\phi\right\rangle_{L^2(\Theta)}\rightarrow \left\langle v(\cdot),\phi\right\rangle_{L^2(\Theta)}\quad \text{in} \; L^{r}(0,T;\mb{R}).
\]
Note that on bounded sets in $L^{\infty}(0,T;L^{2}(\Theta))$, the topology of $L^{r}(0,T;\ltw)$ is identical to the one defined by the above metric.

Moreover, using similar arguments as in the proof of  \cite[Lemma A.3]{droniou2020design}, we have the following result.

\begin{lemma}\label{lem:hb}  Let $\beta\in (0,1)$. For any normed space $V$ and $v\in{H^{\beta}(0,T; V)}$ there exists $C$ depending only on $v$ and $\beta$ such that
	\beas
	  \|v(\cdot+s)-v\|_{L^{2}(0,T-s;V)}\le C|s|^{\beta}\|v\|_{H^{\beta}(0,T; V)}\quad\forall s\in[0,T].
	\eeas
\end{lemma}
In the following lemma, we prove the tightness of the sequence in appropriate product space.
\begin{lemma}
	The laws of $\Rm$ in $\mc{E}$ are tight.
 \end{lemma}
	\begin{proof}
		Consider, for a fixed constant $\lambda$, the sets
		\beas 
		  K_{m}(\lambda):=&\Bigl\{ \nu\in \pdm\xdo[_{m}]^{N_{m}+1}:\exists w\in \xdo[_{m}]^{N{m}+1} \;\text{satisfying}\; \pdm w=\nu,\;\nrb w\nre[\beta,\mc{D}_{m}]\le \lambda
		\Bigr\} 
		\eeas
  and denote
		$$
		  \mathfrak{K}(\lambda):=\bigcup_{m\in \mb{N}}K_{m}(\lambda).
		$$ 
		For each $m\in\mb{N}$, $K_{m}(\lambda)$ is bounded in the finite dimensional space $\pdm\xdo[_{m}]^{N+1}$ and therefore relatively compact in $L^{2}(0,T;L^{2})$. We set $X_{m}:=\pdm \xdo[_{m}]^{N+1}$ and define the norm 
		\begin{equation}\label{eq:def.norm.Xm}
		  \nrb \nu\nre[X_{m}]=\min \left\lbrace \nrb \ddm w\nre[L^2(\Theta)]: w\in \xdo[_{m}]^{N+1}\; \text{such that}\; \pdm w=\nu  \right\rbrace.
		\end{equation}
		By the compactness of $(\mc{D}_{m})_{m\in\mb{N}}$, $(X_{m})_{m\in\mb{N}}$ is compactly embedded in $L^{2}$ in the sense of \cite[Definition C.4]{droniou2018gradient}.
  
        Let $(\nu_{m})_{\minn}$ be a sequence such that $\nu_{m}\in K_{m}(\lambda)$ for all $m$ and, by definition of that space, take $w_m\in\xdo[_{m}]^{N{m}+1}$ such that $\pdm w_{m}=\nu_m$ and $\nrb w_{m}\nre[\beta,\mc{D}_{m}]\le \lambda$. By definitions \eqref{eq:dmnorm} and \eqref{eq:def.norm.Xm} of the norms $\nrb\cdot\nre[\beta,\mc{D}_m]$ and $\nrb\cdot\nre[X_m]$ we have
		\[
		\nrb \nu_{m}\nre[L^{1}(0,T;X_{m})]\le \sqrt{T}\nrb \nu_{m}\nre[L^2(0,T;X_{m})]\le\sqrt{T}\nrb\dd w_m\nre[L^2(\Theta_T)]\le \sqrt{T}\lambda.
		\]		
		Using Lemma \ref{lem:hb} and the fact that $\nrb \nu_m\nre[H^\beta(0,T;L^2(\Theta))]\le \nrb w_m\nre[\beta,\mc{D}_m]\le\lambda$, we have a bound on time translates of $\nu_m$ that is uniform in $m\in\mb{N}$. Therefore, using \cite[Proposition C.5]{droniou2018gradient}, any sequence $(\nu_{m})_{m\in\mb{N}}$
		such that $\nu_{m}\in\ K_{m}(\lambda)$ for each $m\in{\mb{N}}$ 
		   is relatively compact in $L^{2}(0,T;L^{2})$. Combining this with the relative compactness of each $K_m$ and invoking \cite[Lemma A.4]{droniou2020design} proves that $\mathfrak{K}(\lambda)$ is relatively compact in $L^2(0,T;L^2(\Theta))$. Moreover, using the Chebyshev inequality and \eqref{eq:dmzeta} for any $m\in{\mb{N}}$, we have
        \begin{equation*}
		\mathcal{L}\left( \pdm[m] \zeta(u_{m})\right)\left( L^{2}(0,T;L^{2})\setminus \mathfrak{K}(\lambda)\right)  \le\mathbb P\{ \pdm[m]\zeta(u_{m})\notin K_{m}(\lambda)\}
            =
		\mathbb P\left( \nrb \zeta(u_{m})\nre[\beta,\mc{D}_{m}]> \lambda
		 \right) 
		 \le \frac{C_{f,T,\mc{Q},u_{0},\Lambda}}{\lambda}.	
		\end{equation*}
    As the right-hand side does not depend on $m$ and tends to $0$ as $\lambda\to\infty$, and since $\mathfrak{K}(\lambda)$ is compact, this concludes the proof that the laws of $(\pdm \zeta(u_{m}))_{m\in \mathbb N}$ are tight in $L^2(0,T;L^{2}(\Theta))$.
  
    In addition, from \cite[lemma A.3]{droniou2020design}, $\rm{W}^{\beta,4}(0,T,L^{2})\cap\linf\xhookrightarrow[]{c}\lrw$ and $\hbt\cap L^{\infty}(0,T)\xhookrightarrow[]{c}L^{4}(0,T)$. Therefore, using the bounds on $\left( \pdm u_{m}\right) _{m\in \mb{N}}$, $\left( \left(\ips \pdm u_{m},\pdm\Pdm \phi_{i}\ipe \right)_{\minn[i]}\right)_{\minn} $   and $\left( \mdm\right) _{m\in \mb{N}}$ stated in Remark \ref{rm:ubound}, Lemmas \ref{lem:hbt} and \ref{t:4} imply the tightness of the law of $\Rm$ in $\mc{E}$.
    \end{proof}

In the following lemma, we show the almost sure convergence of $\Rm$ up to a change of probability space using Skorokhod theorem.

\begin{lemma}
	There exists a probability space $(\ol{\Omega},\ol{\mc{F}},\ol{\mb{P}})$, a sequence of random variables $(\ol{u}_{m}, \ol{M}_{m}, \ol{W}_{m})_{m\in\mb{N}}$ and random variables $(\ol{u},\ol{M},\ol{W})$ on this space such that
	\begin{itemize}
		\item $\ol{u}_{m}\in \xdo[_{m}]^{N_m+1}$ for each $m\in\mb{N}$,
		\item The laws of $$(\pdm \ol{u}_{m}, \pdm\zeta\left(\ol{u}_{m}\right), \ddm \zeta (\ol{u}_{m}), \ol{M}_{m}, \ol{W}_{m},\left( \ips\pdm \ol{u}_{m},\pdm\Pdm \phi_{i}\ipe \right)_{\minn[i]})$$ and $\rmm$ coincide for each $m\in{\mb{N}}$,
		\item $(\ol{u}, \ol{M}, \ol{W})$ takes its values in $L^{2}(0,T;L^{2})\times \lrw\times C([0,T];L^{2})$,
		\item up to a subsequence as $m\rightarrow \infty$
		\beas \pdm \ol{u}_{m}\rightarrow \ol{u} \quad \text{a.s.\ in} \quad \lilws, \eeas
		\bea\label{eq:zetaas} \pdm \zeta(\ol{u}_{m})\rightarrow \zeta(\ol{u}) \quad \text{a.s.\ in} \quad \ltl, \eea
		\bea\label{eq:dzetas} \ddm \zeta(\ol{u}_{m})\rightarrow \nabla \zeta(\ol{u}) \quad \text{a.s.\ in} \quad \ltldw, \eea
		\bea\label{eq:rvMm}  \ol{M}_{m}\rightarrow \ol{M} \quad \text{a.s.\ in} \quad \lrw, \eea
		\bea[eq:rvWm]  \ol{W}_{m}\rightarrow \ol{W} \quad \text{a.s.\ in} \quad \ctl, \eea
		\bea\label{eq:tas}  \ips\pdm \ol{u}_{m}, \phi_{i}\ipe\rightarrow \ips\ol{u}, \phi_{i}\ipe
		\; \quad \text{a.s.\ in} \quad L^{4}(0,T)\;\text{for each}\; \minn[i], \eea
		\item $\ol{u}_m$ is a solution to the gradient scheme \eqref{eq:gs} with $\mc{D}=\mc{D}_{m}$ and $W$ is replaced by $\ol{W}_{m}$.
	\end{itemize}
Furthermore, up to a subsequence as $m\rightarrow \infty$, for almost all $t \in (0,T)$ and $r<4$
\bea\label{eq:ttas} \ips\pdm \ol{u}_{m}(t), \phi_{i}\ipe\rightarrow \ips\ol{u}(t), \phi_{i}\ipe\; \text{in}\; L^{r}(\ol{\Omega}) \;\text{for each}\; \minn[i],  \eea
\bea\label{eq:martingaletscon} \ol{M}_{m}(t) \rightarrow \ol{M}(t)\quad\; \text{in}\;\quad L^{r}(\ol{\Omega}; L^{2}_{\text{w}}). 
\eea
\end{lemma}

\begin{proof}
	From the Jakubowski version of Skorokhod theorem %
 (see Theorem \ref{thm:J_Skorokhod})
 , we get a probability space $(\ol{\Omega},\ol{\mc{F}},\ol{\mb{F}},\ol{\mb{P}})$ with filtration $\ol{\mc{F}}_{t}:=\sigma \{\ol{W}(s),0\le s\le t\}$, a sequence of random variables 
  $$
	  (\ol{y}_{m},\ol{z}_{m},\ol{Z}_{m},\ol{M}_{m},\ol{W}_{m},\left( \ol{\mc{P}}_{m,i}\right) _{\minn[i]})
	$$
	on this space and taking values in the space $\mc{E}$ with the same laws, for each $m\in{\mb{N}}$, as $\mc{R}_m$ and random variables $(\ol{u},\ol{z},\ol{Z},\ol{M},\ol{W},\left( \ol{\mc{P}}_{i}\right)_{\minn[i]})$ in $\mc{E}$ such that there exists a set $A\subset \ol{\Omega}$ satisfying $\ol{\mb{P}}(A)=1$ and, up to a subsequence, for all $\omega \in A$ as $m\rightarrow \infty$,
	\bea[eq:rvum]
	\ol{y}_{m}(\omega)\rightarrow \ou(\omega)\quad \text{in} \quad \lilws,
	\eea
        \beas\label{eq:rvzm}
	\ol{z}_{m}(\omega)\rightarrow \ol{z}(\omega)\quad \text{in} \quad \ltl,
	\eeas
        \beas\label{eq:rvZM}
        \ol{Z}_{m}(\omega)\rightarrow \ol{Z}(\omega)\quad \text{in} \quad \ltldw,
        \eeas   
\bea\label{eq:rvpi}
 \ol{\mc{P}}_{m,i}(\omega)\rightarrow \mc{\ol{P}}_{i}(\omega) \quad \text{in} \quad L^{4}(0,T)\;\text{for each}\; \minn[i],
\eea
	and the convergence of \eqref{eq:rvMm} and \eqref{eq:rvWm} hold. Since the laws of $(\ol{y}_{m},\ol{z}_{m},\ol{Z}_{m},\left(  \ol{\mc{P}}_{m,i}\right)_{\minn[i]})$ and 
	$$
	\left( \pdm u_{m}, \pdm\zeta\left(u_{m}\right), \ddm \zeta (u_{m}),\left(\ips \pdm u_{m},\pdm\Pdm \phi_{i}\ipe \right)_{\minn[i]}\right) 
	$$ 
	are identical, there exists $\ou_{m}\in\xdo[_{m}]$ such that
	\bea[eq:limmatch] 
	  &\ol{y}_{m}(\omega)=\pdm \ou_{m}(\omega), \;\ol{z}_{m}(\omega)=\pdm \zeta(\ou_{m}(\omega)), \;\ol{Z}_{m}(\omega)=\ddm \zeta(\ou_{m}(\omega)), \\
	  &\ol{\mc{P}}_{m,i}(\omega)= \ips\pdm \ou_{m}(\omega),\pdm\Pdm \phi_{i}\ipe\; \text{for each}\; \minn[i], 
	\eea
	and $\ou_{m}$ satisfies \eqref{eq:gs} a.s., with $\mc{D}=\mc{D}_{m}$ and $W$ replaced by $\ol{W}_{m}$. Moreover, \cite[Lemma 4.8]{droniou2018gradient} gives	
        \beas
	\ol{Z}(\omega)=\nabla\ol{z}(\omega).
	\eeas
	Using \eqref{eq:rvum}-\eqref{eq:limmatch}, we have $\pdm \ou_{m}(\omega) \rightarrow \ou(\omega)$ weakly in $\ltl$ and $\pdm \zeta(\ou_{m}(\omega))\rightarrow \ol{z}(\omega)$ strongly in $\ltl$.  Since $\zeta$ is non decreasing, we use the Minty trick \cite[Lemma 8.2]{doi:10.1137/19M1260165} to deduce that $\zeta(\ou(\omega))=\ol{z}(\omega)$ a.e.\ in $\Theta_{T}$. Therefore, as a conclusion from the above discussion, we get 
	\beas 
	 \pdm \zeta(\ou_{m}(\omega))\rightarrow \zeta(\ou(\omega))\; \text{in}\; \ltl,\; \text{and}\; \ddm \zeta(\ou_{m}(\omega))\rightarrow \nabla\zeta(\ou(\omega))\; \text{\;in}\; \ltldw.
	 \eeas
    To prove \eqref{eq:tas}, we first observe using \eqref{eq:rvpi}, weak-strong convergence and the definition of consistency of space-time GDs that 
	\bea\label{eq:umpdmphiC}
	  \ips\pdm \ou_{m},\pdm\Pdm \phi_{i}\ipe\rightarrow \ips\ol{u}, \phi_{i}\ipe \quad \text{a.s.\ in} \quad L^{4}(0,T)\;\text{for each}\; \minn[i].
	\eea
	By \eqref{eq:rvum} and \eqref{eq:limmatch}, $(\pdm \ou_{m})_m$ is a.s.\ bounded in $L^\infty(0,T;L^2(\Theta))_{\rm w^*}$ so the definition \eqref{eq:def.SD} of $S_{\mc{D}}$ gives 
	\bea\label{eq:umweakTimeDiffIneq}
	\nrb\ips\pdm \ou_{m}, \phi_{i}\ipe-\ips\pdm \ou_{m},\pdm\Pdm \phi_{i}\ipe\nre[L^{4}(0,T)]^{4}\le C S_{\mc{D}_m}(\phi_i)^4\quad\text{a.s.}
	\eea
	The left hand side converges to $0$ since $S_{\mc{D}_{m}}(\phi_{i})\rightarrow 0$ as $m\rightarrow \infty$, for each $\minn[i]$.
	From \eqref{eq:umweakTimeDiffIneq} and \eqref{eq:umpdmphiC}, we have the convergence \eqref{eq:tas}.
 
        Lastly, to prove \eqref{eq:ttas} and \eqref{eq:martingaletscon}, we observe the following estimate for each $\minn[i]$ using Remark \ref{rm:ubound} and the Cauchy--Schwarz inequality 
	\beas\label{eq:supumphi}
	\sup_{\minn}\expc \nrb\ips\pdm \ou_{m}, \phi_{i}\ipe\nre[L^{4}(0,T)]^{4}\rb\le \sup_{\minn}\expc \nrb\pdm \ou_{m}\nre[\ltlft]^{4}\rb \nrb\phi_i\nre[L^2(\Theta)]^4\le C.
	\eeas
	As a result of \eqref{lem1st2}, the coercivity of $(\mc{D}_{m})_{\minn}$ and \eqref{eq:l41}, we have
	\bea\label{eq:supM2}
	\sup_{\minn}\expc \left\|\pdm \zeta\left(  \ou_{m}\right)\right\|^{4}_{L^{2}(\Theta_{T})}+\left\|\ddm \zeta\left( \ou_{m}\right)\right\|^{4}_{L^{2}(\Theta_{T})}+  \nrb\ol{M}_{m}\nre[L^{\infty}(0,T;L^{2}(\Theta))]^{4}\rb\le C.
	\eea 
	Therefore, for each $\minn[i]$ and $r<4$ the sequence $\left( \ips\pdm \ou_{m}, \phi_{i}\ipe\right) _{\minn}$ is equi-integrable in $L^{r}(\ol{\Omega}\times(0,T))$ and the sequences $\left( \pdm \zeta\left(  \ou_{m}\right)\right)_{\minn} $ and $(\ol{M}_{m})_{\minn}$ are equi-integrable in $L^{r}(\ol{\Omega},L^{2}(\Theta_{T}))$ and $L^{r}(\ol{\Omega}\times(0,T);L^{2}(\Theta))$ respectively. Using \eqref{eq:zetaas}, \eqref{eq:rvMm}, \eqref{eq:tas}, we apply the Vitali theorem to get the following results
	\bea\label{eq:zetaumas}
	\pdm \zeta\left(  \ou_{m}\right) \rightarrow \zeta\left( \ol{u}\right) \;\text{in}\; L^{2}(\ol{\Omega}\times(0,T)\times \Theta)\; \text{as} \; m\rightarrow \infty,
	\eea
	\beas\label{eq:Mmas}
	\ol{M}_{m}\rightarrow \ol{M} \;\text{in}\; L^{r}(\ol{\Omega}\times(0,T);L^{2}_{\text{w}}(\Theta))\; \text{as} \; m\rightarrow \infty,
	\eeas
	\beas\label{eq:umas}
	\ips\pdm \ou_{m}, \phi_{i}\ipe \rightarrow \ips\ol{u}, \phi_{i} \ipe\;\text{in}\; L^{r}(\ol{\Omega}\times(0,T))\; \text{as} \; m\rightarrow \infty \;\text{for}\; \minn[i].
	\eeas
	Hence, there exist a subsequence, still denoted by $(\ol{M}_{m})_{\minn}$, such that $\ol{M}_{m}(t)$ converges to $\ol{M}(t)$ for almost all $t\in (0,T)$ in  $L^{r}(\ol{\Omega};L^{2}_{\text{w}})$, which proves \eqref{eq:martingaletscon}.
	
	Moreover, using the diagonal extraction process, we can find a single subsequence for all $\minn[i]$, still denoted by $\left( \ips\pdm \ou_{m}, \phi_{i}\ipe\right) _{\minn}$, such that $\left( \ips\pdm \ou_{m}(t), \phi_{i}\ipe\right) _{\minn}$ converges to $\ips\ol{u}(t), \phi_{i} \ipe$ for almost all $t\in (0,T)$ in $L^{r}(\ol{\Omega})$. This implies the convergence \eqref{eq:ttas}.
\end{proof}

The continuity of the stochastic processes $\ou$ and $\oM$ are shown in the following lemma.
\begin{lemma}\label{lemmaMUCont}
The stochastic processes $\ou$ has a continuous version in $C([0,T],L^{2}_{\text{w}}(\Theta))$ and $\ol{M}$ has a continuous version in $C([0,T],L^{2}(\Theta))$.
\end{lemma}

\begin{proof}
	We have, for $r<4$, the following inequality using \eqref{eq:ubound}, \eqref{lam:32} and \eqref{eq:pdmpe}, for $0\le s\le s'\le T$,
	\beas \label{eq:ss'}
	\expc \left|\ips \pdmu (s')-\pdmu (s),\phi_{i}\ipe\right|^{r}\rb & \le 2^{3} 	\expc \left|\ips \pdmu (s')-\pdmu (s),\Pdmphi \ipe\right|^{r}\rb\\&+ 2^{3} 	\expc \left|\ips \pdmu (s')-\pdmu (s),\Pdmphi-\phi_{i} \ipe\right|^{r}\rb
	\\ & \le C |s'-s|^{r/2}+C\dtm^{r/2}+C S_{\mc{D}_{m}}(\phi_{i})^r. 
	\eeas
	Using the Jensen inequality, one can write
		$$
		  \dlw(v,w)^{r}\le \sum_{\minn[i]} \frac{\left|\ips v-w,\phi_{i}\ipe \right|^{r} }{2^{i}}
		$$
	to infer
		$$
		\expc \dlw(\pdmu (s'),\pdmu (s))^{r}\rb\le C|s'-s|^{r/2}+C\sum_{\minn[i]} \frac{\dtm^{r/2}+S_{\mc{D}_{m}}(\phi_{i})^r}{2^{i}}.
		$$
	As $m\rightarrow \infty$, the last term tends to $0$ using a discrete version of dominated convergence theorem by observing that $\dtm\rightarrow 0$, $S_{\mc{D}_{m}}(\phi_{i})\rightarrow 0$, $2^{-i}(\dtm^{r/2}+S_{\mc{D}_{m}}(\phi_{i})^r\le 2^{-i}C$ for each $\minn[i]$, and $\sum_{\minn[i]}2^{-i}C=C< \infty$.
    Using \eqref{eq:ttas} and the Fatou lemma, we infer, for almost any $s,s'$
	$$
	\expc \dlw(\ol{u}(s'),\ol{u}(s))^{r}\rb\le C|s'-s|^{r/2},
	$$
	which implies the desired continuity of $\ol{u}$ using Kolmogorov test for $r<4$.
	
	For the continuity of $\ol{M}$, from \eqref{eq:Mtimetrans} and reasoning as in \cite[Lemma 3.5]{droniou2020design}, for $r<4$ we have
	\beas\label{eq:ascoliarzela}
	\expc \nrb \ol{M}_{m}(s')-\ol{M}_{m}(s)\nre^{r}\rb\le C (|s'-s|+\delta t_{\mc{D}_{m}})^{r/2},
	\eeas
	and from \eqref{eq:l41}, we have $\nrb \ol{M}_{m}\nre[L^{\infty}(0,T;L^{r}(\ol{\Omega};L^{2}(\Theta)))]\le C$. This implies that, for all $s\in [0,T]$, $\left\lbrace \ol{M}_{m}(s): \minn \right\rbrace $ is relatively compact in $L^{r}(\ol{\Omega};L^{2}(\Theta))_{\rm w}$. Therefore, using the discontinuous Ascoli--Arzela theorem \cite[Proposition C11]{droniou2018gradient}, we have
	 $$\ol{M}_{m}\rightarrow \ol{M}\; \text{uniformly on}\; [0,T]\; \text{in} \; L^{r}(\ol{\Omega};L^{2}(\Theta))_{\text{w}}\; \text{as}\; m\rightarrow \infty,$$ and $\ol{M}\in C([0,T], L^{r}(\ol{\Omega};L^{2}(\Theta))_{\text{w}})$. It follows form \eqref{eq:Mtimetrans} and Fatou's lemma that
	 $$\expc \nrb \oM (s')-\oM(s)\nre^{r}\rb\le \liminf_{m}\expc \nrb \ol{M}_{m}(s')-\ol{M}_{m}(s)\nre^{r}\rb\le C |s'-s|^{r/2}.$$
	 Applying the Kolmogorov test, we have the continuity of $\oM$.
	 \end{proof}

\section{Identification of the limit}\label{sec:identify_limit}
		 
		 In this section, we will first find the representation of the martingale $\ol{M}$ and then prove the main theorem. We know that $\ol{M}$ is continuous and square integrable from Lemma \ref{lemmaMUCont} and \eqref{eq:supM2} respectively. Following the same arguments as in the proof of \cite[Lemma 5.1]{droniou2020design} for $f(\zeta( \ol{u}))$ using \eqref{eq:zetaas}, \eqref{eq:martingaletscon} and \eqref{eq:zetaumas}, we have quadratic variation of $\ol{M}$ defined for all $a,b\in L^{2}(\Theta)$ by
		 \beas\label{eq:qvM}
		 \left\langle \left\langle \ol{M}(t)\right\rangle \right\rangle (a,b)=\int_{0}^{t} \left\langle \left(f(\zeta (\ol{u}))\mc{Q}^{1/2}\right)^{*}(a),\left(f(\zeta (\ol{u}))\mc{Q}^{1/2}\right)^{*}(b) \right\rangle_{\mc{K}}ds 
		 \eeas
		 for any $t\ge0$. Therefore, applying the continuous martingale representation theorem \cite[Theorem 8.2]{Prato2014stochastic}, there
        exists a probability space $(\hat{\Omega}, \hat{\mc{F}},\hat{\mb{P}})$, a filtration $\{\hat{\mc{F}_{t}}\}$ and a $\mc{Q}$-wiener process $\tilde{W}$ defined on $\left( \tilde{\Omega},\tilde{\mc{F}},\tilde{\mb{P}}\right) :=(\ol{\Omega}\times\hat{\Omega}, \ol{\mc{F}}\times\hat{\mc{F}},\ol{\mb{P}}\times\hat{\mb{P}})$ adapted to $\{\tilde{\mc{F}_{t}}:=\ol{\mc{F}_{t}}\times\hat{\mc{F}_{t}}\}$ such that $\ol{M}(\cdot,\ol{\omega})=\tilde{M}(\cdot, \ol{\omega},\tilde{\omega})$ and $\ol{u}(\cdot,\ol{\omega})=\tilde{u}(\cdot, \ol{\omega},\tilde{\omega})$ a.s.\ in $(\ol{\omega},\tilde{\omega})$ and we have for every $t\ge0$,
		 \bea\label{eq:Mrep}
		 \tilde{M}(t,\cdot)=\int_{0}^{t}f(\zeta(\tilde{u}(s,\cdot)))d \tilde{W}(s).
		 \eea
		 
		 We are now ready to prove our main theorem.

		 \begin{proof}[Proof of Theorem \ref{th:mt}] For any $t\in[0,T]$ and for each $\minn$, there exist $k\in \{0,\cdots ,N_{m}-1\}$ such that $t\in (\tn[k],\tn[k+1]]$. Recalling that $\ou_m$ solves the gradient scheme with $W$ replaced by $\ol{W}_m$, we have
		 	\[
		 	\ip[\ddhm \tilde{u}_{m}, \pdm\phi]+\dtm\ip[\lpdo\ddm \zeta(\tilde{u}^{(n+1)}_{m}),\ddm\phi]=\ip[f(\pdm \zeta(\tilde{u}^{(n+1)}_{m}))\lwm,\pdm\phi].
      \]
		 Summing this relation from $n=0$ to $n=k$ and choosing $\phi:=\Pdm \psi$, where $\psi\in H_{0}^{1}(\Theta)$, we obtain,
	 \bea\label{eq:lastgs}
	 \ip[\pdm \tilde{u}_{m}(t), \Pdm\pdm\psi]- \ip[\pdm\uo, \Pdm\pdm\psi]&+\sum_{n=0}^{k}\dtm\ip[\lpdo\ddm \zeta(\tilde{u}^{(n+1)}_{m}),\ddm\Pdm\psi]\\&=\ip[\tilde{M}_{m}(t),\pdm\Pdm\psi].
	 \eea
	 Using \eqref{eq:ttas}, \eqref{eq:martingaletscon}, $\Pdm \psi \rightarrow\psi$ in $L^{2}(\Theta)$ and the consistency of $(\mc{D}_{m})_{\minn}$, we obtain for almost every $t$,
	 \bea\label{eq:uc}
	 \ip[\pdm \tilde{u}_{m}(t), \Pdm\pdm\psi]\rightarrow \ip[ \tilde{u}(t), \psi]\quad \text{in}\quad L^{2}(\tilde{\Omega})\\
	 \ip[\pdm\uo, \Pdm\pdm\psi]\rightarrow \ip[u_{0}, \psi]\quad \text{in}\quad L^{2}(\tilde{\Omega})\\
	 \ip[\tilde{M}_{m}(t),\pdm\Pdm\psi]\rightarrow \ip[\tilde{M}(t),\psi] \quad \text{in}\quad L^{2}(\tilde{\Omega}).
	 \eea
	 To prove the convergence of the last term of left-hand side of \eqref{eq:lastgs}, we first observe that
	 \bea\label{eq:sumlastlhs}
	 \sum_{n=0}^{k}\dtm\ip[\lpdo\ddm \zeta(\tilde{u}^{(n+1)}_{m}),\ddm\Pdm\psi]&=\int_{0}^{t}\ip[\lpdo\ddm\zeta(\tilde{u}_{m}(s)),\ddm\Pdm\psi]ds\\& +\int_{t}^{\lceil t/\dtm\rceil \dtm}\ip[\lpdo\ddm\zeta(\tilde{u}_{m}(s)),\ddm\Pdm\psi]ds.
	 \eea
	Using the convergence \eqref{eq:dzetas} and $\ddm \Pdm \psi \rightarrow\nabla \psi$ in $L^{2}(\Theta)$, we have the convergence of the first term of the right-hand side of \eqref{eq:sumlastlhs}, for any $t\in [0,T]$
	\bea\label{eq:gdzetc}
	\int_{0}^{t}\ip[\lpdo\ddm\zeta(\tilde{u}_{m}(s)),\ddm\Pdm\psi]ds\rightarrow \int_{0}^{t}\ip[\lpdo\nabla\zeta(\tilde{u}(s)),\nabla\psi]ds.
	\eea
   Lastly, we note from the following inequality that the expectation of the absolute value last term in the right-hand side of \eqref{eq:sumlastlhs} tends to zero as $m\rightarrow\infty$. 
   \beas
   &\expc \left|\int_{t}^{\lceil t/\dtm\rceil \dtm}\ip[\lpdo\ddm\zeta(\tilde{u}_{m}(s)),\ddm\Pdm\psi]ds\right|\rb\\& \le
   \expc \int_{t}^{\lceil t/\dtm\rceil \dtm}\nrb\lpdo\ddm\zeta(\tilde{u}_{m}(s))\nre\nrb\ddm\Pdm\psi\nre ds\rb\\& \le
   \omu C\dtm^{1/2}\expc \left( \int_{0}^{T}\nrb\ddm\zeta(\tilde{u}_{m}(s))\nres ds\right)^{1/2}\rb\\& \le
   \omu C\dtm^{1/2},
   \eeas
   where the conclusion comes from \eqref{lem1st1}.
  Using \eqref{eq:uc}-\eqref{eq:gdzetc} and \eqref{eq:Mrep}, we pass to the limit in \eqref{eq:lastgs} to observe that $\tilde{u}$ satisfies (4) in Definition \ref{def:maindefinition}. This shows that $(\tilde{\Omega},\tilde{\mc{F}},\tilde{\mb{F}},\tilde{\mb{P}},\tilde{u}(\cdot),\tilde{W}(\cdot))$ is a weak martingale solution.
  \end{proof}

  	\begin{remark}[Strong convergence of the gradient]
  	In the deterministic case, we can prove a uniform-in-time strong-$L^2$ convergence of $\pdm \tilde{u}_{m}$ and a strong-$L^2$ convergence of $\ddm\zeta(\tilde{u}_{m})$, see \cite[Theorems 2.12 and 2.16]{droniou2016uniform}. These convergences are based on an energy equality for the continuous solution, and start from taking the superior limit of the deterministic version of \eqref{eq:to2}. There are however several challenges to applying this approach here. First, we would need a stronger form of the notion of weak solution \eqref{eq:notion.ws} that would allow us to take a random time-dependent test function $\psi$ (to then use $\psi=\zeta(\tilde{u})$ as test function and establish the energy equality for the continuous solution). Second, passing to the limit in the stochastic term of \eqref{eq:to2} does not seem straightforward: as shown in the proof of Lemma \ref{eq:lemma1}, handling this term requires to introduce the difference $\zeta(u^{(n+1)})-\zeta(u^{(n)})$, which results in a term that is only bounded (does not necessarily vanish in the limit) -- see the reasoning that leads to \eqref{eq:t13}; as a	 consequence, the resulting discrete energy estimate is only an upper bound with constants that do not necessarily correspond to those in the continuous energy, which prevents an application of the technique in \cite{droniou2016uniform}. Adapting the approach in this work to stochastic PDEs however remains an interesting research direction.
   	\end{remark}

\section{Numerical Examples}\label{sec:numerical_examples} 

We consider the stochastic Stefan problem with dimension $d=2$, $\Theta=(0,1)^{2}$, $T=1$,
\begin{equation*}
	\zeta(u) =	
	\begin{cases}
		u, & \text{if } u \leq 0\\
		1, & \text{if } 0\le u\leq 1\\
		u-1, & \text{if } 1\le u
	\end{cases}\quad\text{and}\quad f(\zeta(u))=\nf \Xi(u)^{1/2}	
\end{equation*}
where $\nf$ is a constant and represents the noise factor.
The analytical solution for the deterministic non-homogeneous Stefan problem \cite{droniou2014uniform}, i.e.,
\begin{equation*}
	u(x_{1},x_{2},t) =	
	\begin{cases}
		2\exp (t-x_{1})\quad (>2), & \text{if } x_{1} < t\\
		\exp (t-x_{1})\quad (<1), & \text{if } t<x_{1},		
	\end{cases}	
\end{equation*}
 is used to fix the initial and boundary conditions for the "Test-1". In the "Test-2", we use $u(0,\cdot)=2$ and $\zeta(u)=-1$ on $(0,T)\times \partial \Theta$.
 We choose $\dt=h^{2}$ to ensure that the truncation in time is not dominating the error and spatial truncation error remains the leading term in the estimates.  Moreover, Gradient Scheme \eqref{eq:gs} is nonlinear and thus requires a nonlinear iterative method to determine the approximate solution. In such a case, the Newton method is a common choice due to its quadratic convergence. At each time step, the initial guess in the Newton algorithm is the solution computed at the previous time step; selecting $\dt=h^2$ gives a level of certainty (even higher for the finer meshes)  that $\dt$ is small enough so that this initial guess is close enough to the actual solution of the nonlinear system, and thus that the Newton algorithm converges.
 
 The Wiener processes are simulated in advance on the finest time scale, and used for all time scales (note that our time discretisations are hierarchical: the finest scales are sub-scales of the coarser ones). This ensures that we do not re-simulate different processes, and thus different solutions, each time we refine the mesh. 

\begin{table}[h!]
\centering
\caption{Data for the triangular meshes}%
\label{tab:tri_mesh}
\begin{tabular}{|c|c|c|c|c|}
\hline
 Mesh& Size& Nb. Cells & Nb. Edges&Nb. Vertices\\
\hline
    mesh1-01&    0.250&   56&     92&     37\\
    mesh1-02&    0.125&   224&    352&    129\\
    mesh1-03&    0.063&   896&    1376&   481\\
    mesh1-04&    0.050&   1400&   2140&   741\\
    mesh1-05&    0.031&   3584&   5440&   1857\\
    mesh1-06&    0.016&   14336&  21632&  7297\\
\hline	
	\end{tabular}%
 \end{table}%
 \begin{table}[h!]
\centering
\caption{Data for the hexagonal meshes}%
\label{tab:hexa_mesh}
\begin{tabular}{|c|c|c|c|c|}
\hline
 Mesh& Size& Nb. Cells & Nb. Edges&Nb. Vertices\\
\hline
    hexa1-01&    0.241&   121&     400&     280\\    
    hexa1-02&    0.130&   441&    1400&   960\\
    hexa1-03&    0.093&   841&    2632&    1792\\
    hexa1-04&    0.065&   1681&    5200&   3520\\
    hexa1-05&    0.033&   6561&   20000&   13440\\
\hline	
\end{tabular}%
 \end{table}%
 
The MLP1 and the modified HMM methods, as described in Section \ref{sec:examplesGS}, are used to run the tests. Their codes are available at \url{https://github.com/jdroniou/matlab-SSP}. HMM, being a polytopal method, is simulated over triangular and hexagonal families of meshes whereas MLP1 is tested only on triangular meshes. These meshes can be found in the above mentioned repository; Tables \ref{tab:tri_mesh} and \ref{tab:hexa_mesh} provide some mesh data. In all of our tests, the Newton method works usually quite well; it however needs to be relaxed some times. For the HMM scheme, on average, 2 to 3 relaxations are required for the coarsest mesh (triangular or hexagonal), and about one for the finer meshes; about one relaxation is required for each mesh when using the MLP1 scheme. We also observed that, on average and for each time step, 3 to 4 Newton iterations are required for the MLP1 scheme. However, for HMM, this average is in between 4 to 15 iterations (the larger iterations are for the coarsest mesh) in case of triangular meshes and 7 to 22 iterations in case of hexagonal meshes. 
%

\subsection{Accuracy tests}

To assess the accuracy of scheme and validate key theoretical estimates on $\pd \zeta(u)$, $\dd \zeta(u)$ and $\pd \Xi(u)$, we compute the errors, for each of these quantities of interest, between the reference quantity computed on the finest mesh of each family, and the interpolate on the finest mesh of the quantity computed on each mesh of the family. In case of the MLP1 scheme, the interpolation is performed using the values of the coarse $\mb{P}^1$ piecewise linear function at the vertices of the fine mesh. For HMM, we use the following algorithm to compute the interpolate $I_R w$ of a quantity $w\in \xdo$ on the finest mesh:
\begin{enumerate}
\item For each $x^f$ center of a fine cell or edge, find a coarse mesh $K^c$ that contains $x^f$.
\item Set the value $(I_Rw)_f$ on this fine cell or edge by a linear interpolation of $w$ in $K^c$, that is:
\[
(I_R w)_f=w_{K^c}+\nabla_{K^c}w\cdot (x^f-x_{K^c}).
\]
\end{enumerate}
This algorithm is suitable for quantities $w$ which are expected, from the model, to have a gradient -- in our case, this means $w=\zeta(u)$. However, when $w$ corresponds to a quantity whose gradient may not be defined (e.g., $w=\Xi(u)$), this interpolation algorithm can lead to very bad values; in that case, we resort to a simpler interpolation by setting $(I_Rw)_f=w_{K^c}$.

The following relative errors based on the averages of over 100 random simulations of the Brownian motion are used in the comparison plots:
\begin{align*}
	E^{m}_{\pd\zeta}={}&\fr{\expc \nrb \pd\zeta(u_{R})-\pd I_{R}(\zeta(u_{m}))\nre[L^{2}(\Theta_{T})]^{2}\rb^{\fr{1}{2}}}{\expc \nrb \pd\zeta(u_{R})\nre[L^{2}(\Theta_{T})]^{2}\rb^{\fr{1}{2}}},\\
	E^{m}_{\dd\zeta}={}&\fr{\expc \nrb \dd\zeta(u_{R})-\dd I_{R}(\zeta(u_{m}))\nre[L^{2}(\Theta_{T})]^{2}\rb^{\fr{1}{2}}}{\expc \nrb \dd\zeta(u_{R})\nre[L^{2}(\Theta_{T})]^{2}\rb^{\fr{1}{2}}},\\
	E^{m}_{\Xi}={}&\fr{\expc\ith\left|\pd\Xi(u^{(T)}_{R})-\pd I_{R}(\Xi(u^{(T)}_{m}))\right|\rb}{\expc \ith \pd\Xi(u^{(T)}_{R})\rb}
\end{align*}  

 \begin{figure}\centering
\vspace{0.50cm}
  \centerline{ \ref{t1error}}
\begin{minipage}{0.25\textwidth}
    \begin{tikzpicture}[scale=0.63]
        \begin{loglogaxis}[name=p1error,legend columns=3,legend to name=t1error, tick align=outside,tick pos=lower,
        xlabel=h,
         ylabel=\textsc{Error}]
          \logLogSlopeTriangle{0.95}{0.4}{0.1}{1}{black};
        \addplot[green!40!black,thick,mark=+,mark options={color=black,solid},dashed] 
        table[x=h, y=PT_EL2z] {dat/e1p1error.dat};
        \addlegendentry{MLP1: triangular}
        \addplot[blue,thick,mark=*,mark options={color=black,solid}] 
        table[x=h, y=HT_EL2z] {dat/e1hterrors.dat};
        \addlegendentry{HMM: triangular}
         \addplot[red!50!black,thick,mark=diamond,mark options={color=black,solid},dotted] 
         table[x=h, y=HH_EL2z] {dat/e1hherrors.dat};
        \addlegendentry{HMM: hexagonal}
          \end{loglogaxis}
    \end{tikzpicture}
\subcaption{$E^m_{\Pi_D \zeta}$}
\end{minipage}
\hskip 35pt
\begin{minipage}{0.25\textwidth}
    \begin{tikzpicture}[scale=0.63]
        \begin{loglogaxis}[name=plot2
        , tick align=outside,tick pos=lower,
            xlabel=h,
            ylabel=\textsc{Errors}]
             \logLogSlopeTriangle{0.95}{0.4}{0.1}{1}{black};
          \addplot[green!40!black,thick,mark=+,mark options={color=black,solid},dashed] 
          table[x=h, y=PT_EH1z] {dat/e1p1error.dat};
            \addplot[blue,thick,mark=*,mark options={color=black,solid}] 
           table[x=h, y=HT_EH1z] {dat/e1hterrors.dat};
             \addplot[red!50!black,thick,mark=diamond,mark options={color=black,solid},dotted] 
           table[x=h, y=HH_EH1z] {dat/e1hherrors.dat};
         
              \end{loglogaxis}
    \end{tikzpicture}
\subcaption{$E^m_{\nabla_D \zeta}$}
\end{minipage}
\hskip 35pt
\begin{minipage}{0.25\textwidth}
    \begin{tikzpicture}[scale=0.63]
        \begin{loglogaxis}[name=plot3
        , tick align=outside,tick pos=lower,
            xlabel=h,
            ylabel=\textsc{Errors}]
             \logLogSlopeTriangle{0.95}{0.4}{0.1}{1}{black};
          \addplot[green!40!black,thick,mark=+,mark options={color=black,solid},dashed] 
          table[x=h, y=PT_EL1Xi] {dat/e1p1error.dat};
            \addplot[blue,thick,mark=*,mark options={color=black,solid}] 
           table[x=h, y=HT_EL1Xi] {dat/e1hterrors.dat};
             \addplot[red!50!black,thick,mark=diamond,mark options={color=black,solid},dotted] 
           table[x=h, y=HH_EL1Xi] {dat/e1hherrors.dat};
   
        \end{loglogaxis}
    \end{tikzpicture}
\subcaption{$E^m_{\Pi_D \Xi}$}
\end{minipage}
    \caption{Test-1: Errors vs.\ mesh size.}
    \label{fig:T1Evshk2}
 \end{figure}
 \begin{figure}\centering
\vspace{0.50cm}
  \centerline{ \ref{t1normvalues}}
\begin{minipage}{0.25\textwidth}
    \begin{tikzpicture}[scale=0.63]
        \begin{axis}[name=p1values,legend columns=3,legend to name=t1normvalues, tick align=outside,tick pos=lower,
        xlabel=h,
         ylabel=\textsc{Values}]
        \addplot[green!40!black,thick,mark=+,mark options={color=black,solid},dashed] 
        table[x=h, y=PT_L2z] {dat/e1p1values.dat};
        \addlegendentry{MLP1: triangular}
        \addplot[blue,thick,mark=*,mark options={color=black,solid}] 
        table[x=h, y=HT_L2z] {dat/e1htvalues.dat};
        \addlegendentry{HMM: triangular}
         \addplot[red!50!black,thick,mark=diamond,mark options={color=black,solid},dotted] 
         table[x=h, y=HH_L2z] {dat/e1hhvalues.dat};
        \addlegendentry{HMM: hexagonal}
         \end{axis}
    \end{tikzpicture}
\subcaption{Norm of $\Pi_D \zeta$}
\end{minipage}
\hskip 35pt
\begin{minipage}{0.25\textwidth}
    \begin{tikzpicture}[scale=0.63]
        \begin{axis}[name=plot2
        , tick align=outside,tick pos=lower,
            xlabel=h,
            ylabel=\textsc{Values}]
          \addplot[green!40!black,thick,mark=+,mark options={color=black,solid},dashed] 
          table[x=h, y=PT_H1z] {dat/e1p1values.dat};
            \addplot[blue,thick,mark=*,mark options={color=black,solid}] 
           table[x=h, y=HT_H1z] {dat/e1htvalues.dat};
             \addplot[red!50!black,thick,mark=diamond,mark options={color=black,solid},dotted] 
           table[x=h, y=HH_H1z] {dat/e1hhvalues.dat};
              \end{axis}
    \end{tikzpicture}
\subcaption{Norm of $\nabla_D \zeta$}
\end{minipage}
\hskip 35pt
\begin{minipage}{0.25\textwidth}
    \begin{tikzpicture}[scale=0.63]
        \begin{axis}[name=plot3
        , tick align=outside,tick pos=lower,
            xlabel=h,
            ylabel=\textsc{Values}]
          \addplot[green!40!black,thick,mark=+,mark options={color=black,solid},dashed] 
          table[x=h, y=PT_L1Xi] {dat/e1p1values.dat};
            \addplot[blue,thick,mark=*,mark options={color=black,solid}] 
           table[x=h, y=HT_L1Xi] {dat/e1htvalues.dat};
             \addplot[red!50!black,thick,mark=diamond,mark options={color=black,solid},dotted] 
           table[x=h, y=HH_L1Xi] {dat/e1hhvalues.dat};
      
        \end{axis}
    \end{tikzpicture}
\subcaption{Norm of $\Pi_D \Xi$}
\end{minipage}
    \caption{Test-1: Norms of $\pd \zeta(u)$, $\dd \zeta(u)$ and $\pd \Xi(u)$ verses mesh size.}
    \label{fig:T1Vvsh}
 \end{figure}
\begin{figure}\centering
\vspace{0.50cm}
  \centerline{ \ref{T1k1dofnf1}}
\begin{minipage}{0.25\textwidth}
    \begin{tikzpicture}[scale=0.63]
        \begin{loglogaxis}[name=T1k1dofnf11,legend columns=3,legend to name=T1k1dofnf1, tick align=outside,tick pos=lower,
        xlabel=\textsc{Ndofs},
         ylabel=\textsc{Error}]
        \addplot[green!40!black,thick,mark=+,mark options={color=black,solid},dashed] 
        table[x=ndofs, y=PT_EL2z] {dat/e1p1error.dat};
        \addlegendentry{MLP1: triangular}
        \addplot[blue,thick,mark=*,mark options={color=black,solid}] 
        table[x=ndofs, y=HT_EL2z] {dat/e1hterrors.dat};
        \addlegendentry{HMM: triangular}
         \addplot[red!50!black,thick,mark=diamond,mark options={color=black,solid},dotted] 
         table[x=ndofs, y=HH_EL2z] {dat/e1hherrors.dat};
        \addlegendentry{HMM: hexagonal}
         \end{loglogaxis}
    \end{tikzpicture}
\subcaption{Error $E^m_{\Pi_D \zeta}$}
\end{minipage}
\hskip 35pt
\begin{minipage}{0.25\textwidth}
    \begin{tikzpicture}[scale=0.63]
        \begin{loglogaxis}[name=plot2
        , tick align=outside,tick pos=lower,
            xlabel=\textsc{Ndofs},
            ylabel=\textsc{Errors}]
          \addplot[green!40!black,thick,mark=+,mark options={color=black,solid},dashed] 
          table[x=ndofs, y=PT_EH1z] {dat/e1p1error.dat};
            \addplot[blue,thick,mark=*,mark options={color=black,solid}] 
           table[x=ndofs, y=HT_EH1z] {dat/e1hterrors.dat};
             \addplot[red!50!black,thick,mark=diamond,mark options={color=black,solid},dotted] 
           table[x=ndofs, y=HH_EH1z] {dat/e1hherrors.dat};
              \end{loglogaxis}
    \end{tikzpicture}
\subcaption{Error $E^m_{\nabla_D \zeta}$}
\end{minipage}
\hskip 35pt
\begin{minipage}{0.25\textwidth}
    \begin{tikzpicture}[scale=0.63]
        \begin{loglogaxis}[name=plot3
        , tick align=outside,tick pos=lower,
            xlabel=\textsc{Ndofs},
            ylabel=\textsc{Errors}]
          \addplot[green!40!black,thick,mark=+,mark options={color=black,solid},dashed] 
          table[x=ndofs, y=PT_EL1Xi] {dat/e1p1error.dat};
            \addplot[blue,thick,mark=*,mark options={color=black,solid}] 
           table[x=ndofs, y=HT_EL1Xi] {dat/e1hterrors.dat};
             \addplot[red!50!black,thick,mark=diamond,mark options={color=black,solid},dotted] 
           table[x=ndofs, y=HH_EL1Xi] {dat/e1hherrors.dat};
      
        \end{loglogaxis}
    \end{tikzpicture}
\subcaption{Error $E^m_{\Pi_D \Xi}$}
\end{minipage}
    \caption{Test-1: Errors vs.\ Ndofs.}
    \label{fig:T1Evsdof}
 \end{figure}
For Test-1, the error loglog plots versus the mesh sizes are shown in Figure \ref{fig:T1Evshk2}. The errors for $\pd \zeta(u)$ and $\pd \Xi(u)$ seem to decay linearly, while the rate of convergence for the gradient is slightly below 1. We also tested the validity of the theoretical bounds on $\pd \zeta(u)$, $\dd \zeta(u)$ and $\pd \Xi(u)$ by plotting the respective norms of these quantities in Figure \ref{fig:T1Vvsh}. These plots show the convergence of all three schemes to a similar value for each measure. For the convergence on $\pd \zeta(u)$ all three schemes behave in a comparable way, while the HMM scheme slightly outperforms the MLP1 scheme on the other two quantities. In Figure \ref{fig:T1Evsdof} we analyse the convergence rates in terms of the algebraic complexity, by plotting the errors versus the numbers of degrees of freedom (Ndofs); for the HMM scheme, we do not count the cell unknowns in Ndofs since they are locally eliminated by static condensation. These plots show, for a given mesh family, a slight efficiency advantage to MLP1 for the approximation of $\pd\zeta(u)$. This is not unexpected since MLP1 only has vertex degrees of freedom, while HMM has edge degrees of freedom (the interest of this method being its flexibility with respect to the mesh type). We however note that the accuracy vs.~complexity for $\pd\Xi(u)$ are similar, despite a rougher interpolation of this quantity for the HMM scheme. All these results indicate that the HMM scheme on hexagonal meshes performs better in terms of error reduction, while MLP1 is marginally more efficient in terms of number of DOFs. 

\begin{figure}\centering
\vspace{0.50cm}
  \centerline{ \ref{t1errorsep}}
\begin{minipage}{0.25\textwidth}
    \begin{tikzpicture}[scale=0.63]
        \begin{loglogaxis}[name=p1errorsep,legend columns=3,legend to name=t1errorsep, tick align=outside,tick pos=lower,
        xlabel=h,
         ylabel=\textsc{Error}]
          \logLogSlopeTriangle{0.95}{0.4}{0.1}{1}{black};
        \addplot[green!40!black,thick,mark=+,mark options={color=black,solid},dashed] 
        table[x=h, y=PT_EL2z] {dat/p1error.dat};
        \addlegendentry{MLP1: triangular}
        \addplot[blue,thick,mark=*,mark options={color=black,solid}] 
        table[x=h, y=HT_EL2z] {dat/hterrors.dat};
        \addlegendentry{HMM: triangular}
         \addplot[red!50!black,thick,mark=diamond,mark options={color=black,solid},dotted] 
         table[x=h, y=HH_EL2z] {dat/hherrors.dat};
        \addlegendentry{HMM: hexagonal}
          \end{loglogaxis}
    \end{tikzpicture}
\subcaption{$E^m_{\Pi_D \zeta}$}
\end{minipage}
\hskip 35pt
\begin{minipage}{0.25\textwidth}
    \begin{tikzpicture}[scale=0.63]
        \begin{loglogaxis}[name=plot2
        , tick align=outside,tick pos=lower,
            xlabel=h,
            ylabel=\textsc{Errors}]
             \logLogSlopeTriangle{0.95}{0.25}{0.1}{1}{black};
          \addplot[green!40!black,thick,mark=+,mark options={color=black,solid},dashed] 
          table[x=h, y=PT_EH1z] {dat/p1error.dat};
            \addplot[blue,thick,mark=*,mark options={color=black,solid}] 
           table[x=h, y=HT_EH1z] {dat/hterrors.dat};
             \addplot[red!50!black,thick,mark=diamond,mark options={color=black,solid},dotted] 
           table[x=h, y=HH_EH1z] {dat/hherrors.dat};
         
              \end{loglogaxis}
    \end{tikzpicture}
\subcaption{$E^m_{\nabla_D \zeta}$}
\end{minipage}
\hskip 35pt
\begin{minipage}{0.25\textwidth}
    \begin{tikzpicture}[scale=0.63]
        \begin{loglogaxis}[name=plot3
        , tick align=outside,tick pos=lower,
            xlabel=h,
            ylabel=\textsc{Errors}]
             \logLogSlopeTriangle{0.95}{0.4}{0.1}{1}{black};
          \addplot[green!40!black,thick,mark=+,mark options={color=black,solid},dashed] 
          table[x=h, y=PT_EL1Xi] {dat/p1error.dat};
            \addplot[blue,thick,mark=*,mark options={color=black,solid}] 
           table[x=h, y=HT_EL1Xi] {dat/hterrors.dat};
             \addplot[red!50!black,thick,mark=diamond,mark options={color=black,solid},dotted] 
           table[x=h, y=HH_EL1Xi] {dat/hherrors.dat};
   
        \end{loglogaxis}
    \end{tikzpicture}
\subcaption{$E^m_{\Pi_D \Xi}$}
\end{minipage}
    \caption{Test-2: Errors vs.\ mesh size for $\nf=1$.}
    \label{fig:T2EvshK2}
 \end{figure}
\begin{figure}\centering
\vspace{0.50cm}
  \centerline{ \ref{t1normvaluessep}}
\begin{minipage}{0.25\textwidth}
    \begin{tikzpicture}[scale=0.63]
        \begin{axis}[name=p1valuessep,legend columns=3,legend to name=t1normvaluessep, tick align=outside,tick pos=lower,
        xlabel=h,
         ylabel=\textsc{Values}]
        \addplot[green!40!black,thick,mark=+,mark options={color=black,solid},dashed] 
        table[x=h, y=PT_L2z] {dat/p1values.dat};
        \addlegendentry{MLP1: triangular}
        \addplot[blue,thick,mark=*,mark options={color=black,solid}] 
        table[x=h, y=HT_L2z] {dat/htvalues.dat};
        \addlegendentry{HMM: triangular}
         \addplot[red!50!black,thick,mark=diamond,mark options={color=black,solid},dotted] 
         table[x=h, y=HH_L2z] {dat/hhvalues.dat};
        \addlegendentry{HMM: hexagonal}
         \end{axis}
    \end{tikzpicture}
\subcaption{Norm of $\Pi_D \zeta$}
\end{minipage}
\hskip 35pt
\begin{minipage}{0.25\textwidth}
    \begin{tikzpicture}[scale=0.63]
        \begin{axis}[name=plot2
        , tick align=outside,tick pos=lower,
            xlabel=h,
            ylabel=\textsc{Values}]
          \addplot[green!40!black,thick,mark=+,mark options={color=black,solid},dashed] 
          table[x=h, y=PT_H1z] {dat/p1values.dat};
            \addplot[blue,thick,mark=*,mark options={color=black,solid}] 
           table[x=h, y=HT_H1z] {dat/htvalues.dat};
             \addplot[red!50!black,thick,mark=diamond,mark options={color=black,solid},dotted] 
           table[x=h, y=HH_H1z] {dat/hhvalues.dat};
              \end{axis}
    \end{tikzpicture}
\subcaption{Norm of $\nabla_D \zeta$}
\end{minipage}
\hskip 35pt
\begin{minipage}{0.25\textwidth}
    \begin{tikzpicture}[scale=0.63]
        \begin{axis}[name=plot3
        , tick align=outside,tick pos=lower,
            xlabel=h,
            ylabel=\textsc{Values}]
          \addplot[green!40!black,thick,mark=+,mark options={color=black,solid},dashed] 
          table[x=h, y=PT_L1Xi] {dat/p1values.dat};
            \addplot[blue,thick,mark=*,mark options={color=black,solid}] 
           table[x=h, y=HT_L1Xi] {dat/htvalues.dat};
             \addplot[red!50!black,thick,mark=diamond,mark options={color=black,solid},dotted] 
           table[x=h, y=HH_L1Xi] {dat/hhvalues.dat};
      
        \end{axis}
    \end{tikzpicture}
\subcaption{Norm of $\Pi_D \Xi$}
\end{minipage}
    \caption{Test-2: Norms of $\pd \zeta(u)$, $\dd \zeta(u)$ and $\pd \Xi(u)$ vs.\ mesh size for $\nf=1$.}
    \label{fig:T2Vvsh}
 \end{figure}
\begin{figure}\centering
\vspace{0.50cm}
  \centerline{ \ref{T2k1dofnf1}}
\begin{minipage}{0.25\textwidth}
    \begin{tikzpicture}[scale=0.63]
        \begin{loglogaxis}[name=T2k1dofnf1,legend columns=3,legend to name=T2k1dofnf1, tick align=outside,tick pos=lower,
        xlabel=\textsc{Ndofs},
         ylabel=\textsc{Error}]
        \addplot[green!40!black,thick,mark=+,mark options={color=black,solid},dashed] 
        table[x=ndofs, y=PT_EL2z] {dat/p1error.dat};
        \addlegendentry{MLP1: triangular}
        \addplot[blue,thick,mark=*,mark options={color=black,solid}] 
        table[x=ndofs, y=HT_EL2z] {dat/hterrors.dat};
        \addlegendentry{HMM: triangular}
         \addplot[red!50!black,thick,mark=diamond,mark options={color=black,solid},dotted] 
         table[x=ndofs, y=HH_EL2z] {dat/hherrors.dat};
        \addlegendentry{HMM: hexagonal}
         \end{loglogaxis}
    \end{tikzpicture}
\subcaption{Error $E^m_{\Pi_D \zeta}$}
\end{minipage}
\hskip 35pt
\begin{minipage}{0.25\textwidth}
    \begin{tikzpicture}[scale=0.63]
        \begin{loglogaxis}[name=plot2
        , tick align=outside,tick pos=lower,
            xlabel=\textsc{Ndofs},
            ylabel=\textsc{Errors}]
          \addplot[green!40!black,thick,mark=+,mark options={color=black,solid},dashed] 
          table[x=ndofs, y=PT_EH1z] {dat/p1error.dat};
            \addplot[blue,thick,mark=*,mark options={color=black,solid}] 
           table[x=ndofs, y=HT_EH1z] {dat/hterrors.dat};
             \addplot[red!50!black,thick,mark=diamond,mark options={color=black,solid},dotted] 
           table[x=ndofs, y=HH_EH1z] {dat/hherrors.dat};
              \end{loglogaxis}
    \end{tikzpicture}
\subcaption{Error $E^m_{\nabla_D \zeta}$}
\end{minipage}
\hskip 35pt
\begin{minipage}{0.25\textwidth}
    \begin{tikzpicture}[scale=0.63]
        \begin{loglogaxis}[name=plot3
        , tick align=outside,tick pos=lower,
            xlabel=\textsc{Ndofs},
            ylabel=\textsc{Errors}]
          \addplot[green!40!black,thick,mark=+,mark options={color=black,solid},dashed] 
          table[x=ndofs, y=PT_EL1Xi] {dat/p1error.dat};
            \addplot[blue,thick,mark=*,mark options={color=black,solid}] 
           table[x=ndofs, y=HT_EL1Xi] {dat/hterrors.dat};
             \addplot[red!50!black,thick,mark=diamond,mark options={color=black,solid},dotted] 
           table[x=ndofs, y=HH_EL1Xi] {dat/hherrors.dat};
      
        \end{loglogaxis}
    \end{tikzpicture}
\subcaption{Error $E^m_{\Pi_D \Xi}$}
\end{minipage}
    \caption{Test-2: Errors vs.\ Ndofs with $\nf=1$.}
    \label{fig:T2Evsdofk1}
 \end{figure}
The error plots in Figure \ref{fig:T2EvshK2} for Test-2 show, on the contrary, a convergence rate which is higher than one for $\pd \zeta(u)$ and $\pd \Xi(u)$, and of order one for $\dd\zeta(u)$. In these tests, also, the HMM scheme on hexagonal meshes exhibit a smaller error (and sometimes an apparently better rate) compared to the other two schemes. We note however, in Figure \ref{fig:T2Vvsh} (which displays the numerical values of the norms of each quantity), that convergence does not seem to be achieved at the considered mesh sizes for $\dd \zeta(u)$; more refinements would probably be required to see the norm of this value start to stagnate around a particular number. In terms of the algebraic complexity, Figure \ref{fig:T2Evsdofk1} shows a similar behaviour as in Test-1. 

\subsection{Stochastic mushy regions}

We then numerically investigate the existence of mushy region (an intermediate region where solid and liquid coexist), denoted below as MR. Since the source term is zero in the deterministic setting for Test-1, we observed (in results not reported here) that there is no mushy region in this case. In the stochastic case, we used Test-2 and estimated the mushy region by computing the expectation and standard deviation, at each time step and using 200 Brownian motions, of the area of $\bigcup\{K\::\:0<(\pd u)|_{K}<1\}$. These expectation and standard deviations are denoted by $\textsc{Exp-MR}$ and $\textsc{SD-MR}$. All the tests here are done using the MLP1 scheme.
\begin{figure}\centering
\vspace{0.50cm}
 \centerline{ \ref{mcmrnf1}}
\begin{tikzpicture}[scale=0.85]
        \begin{axis}[name=mcmrnf11,
        legend columns=5,legend to name=mcmrnf1,
        xlabel=$t$,
        ylabel=\textsc{Exp-mr},
        x post scale =2.5,
        ytick=\empty,
        extra y ticks={0.25, 0.125, 0.063, 0.05, 0.031},
        extra y tick labels={$0.25$, $0.125$, $0.063$, $0.05$, $0.031$},
        ]
         \addplot[blue,thin,mark=none,dashdotdotted]
        table[x=idt, y=nf1_Exp] {mr/M2bm200nf1.txt};
        \addlegendentry{mesh1-02}
        \addplot[violet,thin, mark=none,densely dotted]
        table[x=idt, y=nf1_Exp] {mr/M3bm200nf1.txt};
        \addlegendentry{mesh1-03}
        \addplot[red,thin, mark=none, dashdotted]
        table[x=idt, y=nf1_Exp] {mr/M4bm200nf1.txt};
        \addlegendentry{mesh1-04}
          \addplot[black,thin,mark=none, dashed]
        table[x=idt, y=nf1_Exp] {mr/M5bm200nf1.txt};
        \addlegendentry{mesh1-05}
         \end{axis}
    \end{tikzpicture}
    \caption{Test-2: mesh comparison of expectation of mushy regions vs.\ time with $\nf=1$ and 200 Brownian motions.}
    \label{fig:mrmcnf1}
 \end{figure}%

We performed several experiments to study the existence of possible mushy region with various noise factor and mesh sizes. In the first experiment, we fixed the noise factor to $\nf=1$ and ran the scheme over a set of meshes. The results are plotted in Figure \ref{fig:mrmcnf1}, in which we see that the mushy region seems to initially exist, but then vanishes after a certain time. However, even in the time span where it exists, the area of the mushy region seems to decay with $h$ (compare the areas with the mesh sizes in Table \ref{tab:tri_mesh}), indicating that the mushy region is probably only visible in numerical simulations but do not correspond to an actual mushy region of the continuous model. 
\begin{figure}\centering
\vspace{0.50cm}
 \centerline{ \ref{mrm3bm200}}
\begin{tikzpicture}[scale=0.85]
        \begin{axis}[name=mrm3bm2001,
        legend columns=6,legend to name=mrm3bm200,
        tick align=outside,
        tick pos=lower,
        ytick=\empty,
        extra y ticks={0.25, 0.125, 0.063},
        extra y tick labels={$0.25$, $0.125$, $0.063$},
        xlabel=$t$,
         ylabel=\textsc{Exp-mr},
         x post scale =2.5]
         \addplot[yellow!80!black,thin,mark=none,densely dotted]
        table[x=idt, y=nf1_Exp] {mr/M3bm200nf1.txt};
        \addlegendentry{$\nf=1$}
         \addplot[green!40!black,thin,mark=none,dashdotted]
         table[x=idt, y=nf1000_Exp] {mr/M3bm200nf1000.txt};
         \addlegendentry{$\nf=1000$}
         \addplot[blue,thin,mark=none,dashdotdotted]
         table[x=idt, y=nf2000_Exp] {mr/M3bm200nf2000.txt};
         \addlegendentry{$\nf=2000$}
          \addplot[violet,thin,mark=none,densely dashed]
         table[x=idt, y=nf4000_Exp] {mr/M3bm200nf4000.txt};
         \addlegendentry{$\nf=4000$}
        \addplot[red,thin,mark=none, dashed]
         table[x=idt, y=nf6000_Exp] {mr/M3bm200nf6000.txt};
         \addlegendentry{$\nf=6000$}
        \addplot[black,thin,mark=none,solid]
         table[x=idt, y=nf8000_Exp] {mr/M3bm200nf8000.txt};
         \addlegendentry{$\nf=8000$}
      
         \end{axis}
    \end{tikzpicture}
    \caption{Test-2: expectation of mushy region vs.\ time with Mesh1-03 and $200$ Brownian motions.}
    \label{fig:expc_mrm3bm200}
 \end{figure}
\begin{figure}\centering
\vspace{0.50cm}
 \centerline{ \ref{var_mr}}
\begin{tikzpicture}[scale=0.85]
        \begin{axis}[name=var_mr1,
        legend columns=6,legend to name=var_mr,
        tick align=outside,
        tick pos=lower,
        xlabel=$t$,
         ylabel=\textsc{SD-mr},
         x post scale =2.5]
        \addplot[yellow!80!black,thin,mark=none,densely dotted]
        table[x=idt, y=nf1_SD] {mr/M3bm200nf1.txt};
        \addlegendentry{$\nf=1$}
        \addplot[green!40!black,thin,mark=none,dashdotted]
        table[x=idt, y=nf1000_SD] {mr/M3bm200nf1000.txt};
        \addlegendentry{$\nf=1000$}
         \addplot[blue,thin,mark=none,dashdotdotted]
        table[x=idt, y=nf2000_SD] {mr/M3bm200nf2000.txt};
        \addlegendentry{$\nf=2000$}
        \addplot[violet,thin,mark=none,densely dashed]
        table[x=idt, y=nf4000_SD] {mr/M3bm200nf4000.txt};
        \addlegendentry{$\nf=4000$}
        \addplot[red,thin,mark=none, dashed]
        table[x=idt, y=nf6000_SD] {mr/M3bm200nf6000.txt};
        \addlegendentry{$\nf=6000$}
        \addplot[black,thin,mark=none,solid]
        table[x=idt, y=nf8000_SD] {mr/M3bm200nf8000.txt};
        \addlegendentry{$\nf=8000$}
         \end{axis}
    \end{tikzpicture}
    \caption{Test-2: standard deviation of mushy region vs.\ time with Mesh1-03 with 200 Brownian motions.}
    \label{fig:var_mr3bm200}
 \end{figure}

In the next experiment, the mesh is fixed and the noise is gradually increased, to assess if a larger stochastic forcing term could generate a mushy region. In Figure \ref{fig:expc_mrm3bm200}, we observe a mushy region, for each noise factor, which starts with a larger measure but stabilises as time progress. This larger initial region is probably due to the chosen initial condition, which forces the solution to cross the plateau at the start of the simulation. We also see that the mushy region reduces with the noise factor and almost vanishes when $\nf=1$, except for a small initial time interval (roughly $I_{0.14}:=(0, 0.14)$). The corresponding standard deviation, in Figure \ref{fig:var_mr3bm200}, is negligible (for $\nf=1$). This indicates that the visible mushy region is not an artifact of noise. Further, this area is of order $h$ and we expect that it will reduce with the mesh refinement.%

\begin{figure}\centering
\vspace{0.50cm}
 \centerline{ \ref{mrm4bm200}}
\begin{tikzpicture}[scale=0.85]
        \begin{axis}[name=mrm4bm2001,
        legend columns=6,legend to name=mrm4bm200,
        tick align=outside,
        tick pos=lower,
        ytick=\empty,
        extra y ticks={0.25, 0.125,  0.05},
        extra y tick labels={$0.25$, $0.125$, $0.05$},
        xlabel=$t$,
         ylabel=\textsc{Exp-mr},
         x post scale =2.5]
         \addplot[yellow!80!black,thin,mark=none,densely dotted]
        table[x=idt, y=nf1_Exp] {mr/m4bm200nf1.txt};
        \addlegendentry{$\nf=1$}
        \addplot[green!40!black,thin,mark=none,dashdotted]
        table[x=idt, y=nf1000_Exp] {mr/m4bm200nf1000.txt};
        \addlegendentry{$\nf=1000$}
        \addplot[blue,thin,mark=none,dashdotdotted]
        table[x=idt, y=nf2000_Exp] {mr/m4bm200nf2000.txt};
        \addlegendentry{$\nf=2000$}
         \addplot[violet,thin,mark=none,densely dashed]
        table[x=idt, y=nf4000_Exp] {mr/m4bm200nf4000.txt};
        \addlegendentry{$\nf=4000$}
       \addplot[red,thin,mark=none, dashed]
        table[x=idt, y=nf6000_Exp] {mr/m4bm200nf6000.txt};
        \addlegendentry{$\nf=6000$}
       \addplot[black,thin,mark=none,solid]
        table[x=idt, y=nf8000_Exp] {mr/m4bm200nf8000.txt};
        \addlegendentry{$\nf=8000$}
      
         \end{axis}
    \end{tikzpicture}
    \caption{Test-2: expectation of mushy region vs.\ time with Mesh1-04 and $200$ Brownian motions.}
    \label{fig:expc_mrm4bm200}
 \end{figure}%
\begin{figure}\centering
\vspace{0.50cm}
 \centerline{ \ref{var_mr4bm200}}
\begin{tikzpicture}[scale=0.85]
        \begin{axis}[name=var_mr4bm2001,
        legend columns=6,legend to name=var_mr4bm200,
        tick align=outside,
        tick pos=lower,
        xlabel=$t$,
         ylabel=\textsc{SD-mr},
         x post scale =2.5]
        \addplot[yellow!80!black,thin,mark=none,densely dotted]
        table[x=idt, y=nf1_SD] {mr/M4bm200nf1.txt};
        \addlegendentry{$\nf=1$}
        \addplot[green!40!black,thin,mark=none,dashdotted]
        table[x=idt, y=nf1000_SD] {mr/M4bm200nf1000.txt};
        \addlegendentry{$\nf=1000$}
         \addplot[blue,thin,mark=none,dashdotdotted]
        table[x=idt, y=nf2000_SD] {mr/M4bm200nf2000.txt};
        \addlegendentry{$\nf=2000$}
        \addplot[violet,thin,mark=none,densely dashed]
        table[x=idt, y=nf4000_SD] {mr/M4bm200nf4000.txt};
        \addlegendentry{$\nf=4000$}
        \addplot[red,thin,mark=none, dashed]
        table[x=idt, y=nf6000_SD] {mr/M4bm200nf6000.txt};
        \addlegendentry{$\nf=6000$}
        \addplot[black,thin,mark=none,solid]
        table[x=idt, y=nf8000_SD] {mr/M4bm200nf8000.txt};
        \addlegendentry{$\nf=8000$}
         \end{axis}
    \end{tikzpicture}
    \caption{Test-2: standard deviation of mushy region vs.\ time with Mesh1-04 with 200 Brownian motions.}
    \label{fig:var_mr4bm200}
 \end{figure}%
\begin{figure}\centering
\vspace{0.50cm}
 \centerline{ \ref{mrm5bm200}}
\begin{minipage}{0.45\textwidth}
\begin{tikzpicture}[scale=1]
        \begin{axis}[name=mrm5bm2001,
        legend columns=6,legend to name=mrm5bm200,
        tick align=outside,
        skip coords between index={501}{1025},
        tick pos=lower,
        ytick=\empty,
        extra y ticks={0.25, 0.125, 0.063, 0.05, 0.031},
        extra y tick labels={$0.25$, $0.125$, $0.063$, $0.05$, $0.031$},
        xlabel=$t$,
         ylabel=\textsc{Exp-mr},
         ]
         \addplot[yellow!80!black,thin,mark=none,densely dotted]
        table[x=idt, y=nf1_Exp] {mr/m5bm200nf1.txt};
        \addlegendentry{$\nf=1$}
        \addplot[green!40!black,thin,mark=none,dashdotted]
        table[x=idt, y=nf1000_Exp] {mr/m5bm200nf1000.txt};
        \addlegendentry{$\nf=1000$}
        \addplot[blue,thin,mark=none,dashdotdotted]
        table[x=idt, y=nf2000_Exp] {mr/m5bm200nf2000.txt};
        \addlegendentry{$\nf=2000$}
         \addplot[violet,thin,mark=none,densely dashed]
        table[x=idt, y=nf4000_Exp] {mr/m5bm200nf4000.txt};
        \addlegendentry{$\nf=4000$}
       \addplot[red,thin,mark=none, dashed]
        table[x=idt, y=nf6000_Exp] {mr/m5bm200nf6000.txt};
        \addlegendentry{$\nf=6000$}
       \addplot[black,thin,mark=none,solid]
        table[x=idt, y=nf8000_Exp] {mr/m5bm200nf8000.txt};
        \addlegendentry{$\nf=8000$}
      
         \end{axis}
    \end{tikzpicture}
    \subcaption{Expectation}
\end{minipage}
\hskip 35pt
\begin{minipage}{0.45\textwidth}
\begin{tikzpicture}[scale=1]
        \begin{axis}[name=var_mr5bm2001,
        tick align=outside,
        tick pos=lower,
        xlabel=$t$,
        skip coords between index={501}{1025},
         ylabel=\textsc{SD-mr},
         ]
        \addplot[yellow!80!black,thin,mark=none,densely dotted]
        table[x=idt, y=nf1_SD] {mr/M5bm200nf1.txt};
        \addplot[green!40!black,thin,mark=none,dashdotted]
        table[x=idt, y=nf1000_SD] {mr/M5bm200nf1000.txt};
         \addplot[blue,thin,mark=none,dashdotdotted]
        table[x=idt, y=nf2000_SD] {mr/M5bm200nf2000.txt};
        \addplot[violet,thin,mark=none,densely dashed]
        table[x=idt, y=nf4000_SD] {mr/M5bm200nf4000.txt};
        \addplot[red,thin,mark=none, dashed]
        table[x=idt, y=nf6000_SD] {mr/M5bm200nf6000.txt};
        \addplot[black,thin,mark=none,solid]
        table[x=idt, y=nf8000_SD] {mr/M5bm200nf8000.txt};
         \end{axis}
    \end{tikzpicture}
    \subcaption{Standard deviation}
\end{minipage}
\caption{Test-2: expectation and standard deviation of mushy region vs.\ time with Mesh1-05 and $200$ Brownian motions.}
\label{fig:expc_mrm5bm200}
 \end{figure}%
To sharpen the image, this experiment is repeated on finer meshes ("Mesh1-04" and "Mesh1-05"), with the same set of noise factors (see Figures \ref{fig:expc_mrm4bm200} -- \ref{fig:expc_mrm5bm200}). Comparing Figures \ref{fig:expc_mrm3bm200}--\ref{fig:var_mr3bm200} with \ref{fig:expc_mrm4bm200}--\ref{fig:var_mr4bm200}, we observe that $\textsc{Exp-MR}$ and $\textsc{SD-MR}$ decay as $h$ reduces from $0.063$ to $0.05$. For example, $\textsc{Exp-MR}$ for $\nf=8000$ in Figure \ref{fig:expc_mrm3bm200} remains around $0.25$ but reduces to $0.15$ (roughly) in Figure \ref{fig:expc_mrm4bm200}. Moving forward to the finest mesh, in Figure \ref{fig:expc_mrm5bm200}--(a), the mushy region disappears for each $\nf$ except over $I_{0.14}$, in which its area is of order $h=0.031$. This again tells us that the mushy region tends to vanish  as the mesh size is reduced. The standard deviation plot (Figure \ref{fig:expc_mrm5bm200}--(b)) shows a negligible variability for all noise factors;  the mushy region is therefore not due to an influence of noise but rather to the mesh discretisation.
The behaviours of the expectation and corresponding standard deviation of the mushy regions are summarised in Figures \ref{fig:mrexpplusSDnf8000}--\ref{fig:mrexpplusSDnf1}, which present the expectation and its variability to one standard deviation (above or below) of the mushy regions over time. These figures illustrate that the higher variability of expectation is a consequence of larger noises, and decays as the noise reduces (up to almost vanish for $\nf=1$). The leftover non-zero area of mushy region is deterministic; its area seem to be of order $h$, and therefore vanishes in the limit $h\to 0$.

\begin{figure}\centering
\vspace{0.50cm}
 \centerline{ \ref{nf8000expSD}}
\begin{tikzpicture}[scale=0.85]
        \begin{axis}[name=nf8000expSD1,
        legend columns=5,legend to name=nf8000expSD,
        skip coords between index={501}{1025},
        xlabel=$t$,
        ylabel=\textsc{MR},
        x post scale =2.5,
        ]
         \addplot[blue,thin,mark=none,solid]
        table[x=idt, y=nf8000_Exp] {mr/M3bm200nf8000.txt};
        \addlegendentry{Exp}
         \addplot[violet,thin, mark=none,densely dotted]
         table[x=idt, y expr=\thisrow{nf8000_Exp}+\thisrow{nf8000_SD}] {mr/M3bm200nf8000.txt};
         \addlegendentry{Exp$\pm$ SD}
         \addplot[violet,thin, mark=none,densely dotted]
         table[x=idt, y expr=\thisrow{nf8000_Exp}-\thisrow{nf8000_SD}] {mr/M3bm200nf8000.txt};
         \end{axis}
    \end{tikzpicture}
    \caption{Test-2: expectation and standard deviation of mushy regions vs.\ time with $\nf=8000$ and 200 Brownian motions over Mesh1-03.}
    \label{fig:mrexpplusSDnf8000}
 \end{figure}
\begin{figure}\centering
\vspace{0.50cm}
 \centerline{ \ref{nf4000expSD}}
\begin{tikzpicture}[scale=0.85]
        \begin{axis}[name=nf4000expSD1,
        legend columns=5,legend to name=nf4000expSD,
        skip coords between index={501}{1025},
        xlabel=$t$,
        ylabel=\textsc{MR},
        x post scale =2.5,
        ]
         \addplot[blue,thin,mark=none,solid]
        table[x=idt, y=nf4000_Exp] {mr/M3bm200nf4000.txt};
        \addlegendentry{Exp}
         \addplot[violet,thin, mark=none,densely dotted]
         table[x=idt, y expr=\thisrow{nf4000_Exp}+\thisrow{nf4000_SD}] {mr/M3bm200nf4000.txt};
         \addlegendentry{Exp$\pm$ SD}
         \addplot[violet,thin, mark=none,densely dotted]
         table[x=idt, y expr=\thisrow{nf4000_Exp}-\thisrow{nf4000_SD}] {mr/M3bm200nf4000.txt};
         \end{axis}
    \end{tikzpicture}
    \caption{Test-2: expectation and standard deviation of mushy regions vs.\ time with $\nf= 4000$ and 200 Brownian motions over Mesh1-03.}
    \label{fig:mrexpplusSDnf4000}
 \end{figure}
\begin{figure}\centering
\vspace{0.50cm}
 \centerline{ \ref{nf1000expSD}}
\begin{tikzpicture}[scale=0.85]
        \begin{axis}[name=nf1000expSD1,
        legend columns=5,legend to name=nf1000expSD,
        skip coords between index={501}{1025},
        xlabel=$t$,
        ylabel=\textsc{MR},
        x post scale =2.5,
        ]
         \addplot[blue,thin,mark=none,solid]
        table[x=idt, y=nf1000_Exp] {mr/M3bm200nf1000.txt};
        \addlegendentry{Exp}
         \addplot[violet,thin, mark=none,densely dotted]
         table[x=idt, y expr=\thisrow{nf1000_Exp}+\thisrow{nf1000_SD}] {mr/M3bm200nf1000.txt};
         \addlegendentry{Exp$\pm$ SD}
         \addplot[violet,thin, mark=none,densely dotted]
         table[x=idt, y expr=\thisrow{nf1000_Exp}-\thisrow{nf1000_SD}] {mr/M3bm200nf1000.txt};
         \end{axis}
    \end{tikzpicture}
    \caption{Test-2: expectation and standard deviation of mushy regions vs.\ time with $\nf=1000$ and 200 Brownian motions over Mesh1-03.}
    \label{fig:mrexpplusSDnf1000}
 \end{figure}
\begin{figure}\centering
\vspace{0.50cm}
 \centerline{ \ref{nf1expSD}}
\begin{tikzpicture}[scale=0.85]
        \begin{axis}[name=nf1expSD1,
        legend columns=5,legend to name=nf1expSD,
        skip coords between index={501}{1025},
        xlabel=$t$,
        ylabel=\textsc{MR},
        x post scale =2.5,
        ]
         \addplot[blue,thin,mark=none,solid]
        table[x=idt, y=nf1_Exp] {mr/M3bm200nf1.txt};
        \addlegendentry{Exp}
         \addplot[violet,thin, mark=none,densely dotted]
         table[x=idt, y expr=\thisrow{nf1_Exp}+\thisrow{nf1_SD}] {mr/M3bm200nf1.txt};
         \addlegendentry{Exp$\pm$ SD}
         \addplot[violet,thin, mark=none,densely dotted]
         table[x=idt, y expr=\thisrow{nf1_Exp}-\thisrow{nf1_SD}] {mr/M3bm200nf1.txt};
         \end{axis}
    \end{tikzpicture}
    \caption{Test-2: expectation and standard deviation of mushy regions vs.\ time with $\nf=1$ and 200 Brownian motions over Mesh1-03.}
    \label{fig:mrexpplusSDnf1}
 \end{figure}
These experiments indicate that the visible mushy region is either due to the induced noise or to the space discretisation, and suggest that there is no mushy region for the continuous stochastic problem.

\section{Conclusion}

We presented a generic numerical analysis, based on the GDM framework, of the stochastic Stefan problem driven by a multiplicative noise. Using  discrete functional analysis tools and the Skorokhod theorem, we showed the compactness of the solution to the gradient scheme. We then proved the existence of weak martingale solution and obtained the convergence of the approximate solution. Though these results are available to all of the methods that lies under the hood of the GDM framework, we chose MLP1 and HMM to illustrate them. We observed that, under the influence of multiplicative noise, the overall numerical approximations are reasonably good and corroborate the theoretical results.

\section{Acknowledgement}
This work was partially supported by the Australian Government through the Australian Research
Council's Discovery Projects funding scheme (grant number DP220100937).

\section{Appendix}\label{appen}
Let $(\Omega,\mc{F},\mb{P})$ is a complete probability space and $(E,\mc{B}(E))$ a measurable space where $E$ is separable metric space and $\mc{B}(E)$ denotes the Borel $\sigma$-field on it. Moreover, let $(X_n)_{n\in \mathbb{N}}$ be $E$-valued random variables, that is, measurable functions $X_n:\Omega\rightarrow E$. Every random variable induces the law on $E$ (also called probability measure or distribution on $E$) defined by $\mc{L}(X_n)(A)=\mb{P}(X_n\in A)$ for all $A\in\mc{B}(E)$.
\begin{definition}[Tightness]
Let $E$ be a separable Banach space and let $(X_n)_{n\in \mathbb{N}}$ be $E$-valued random variables. The laws of $(X_n)_{n\in\mathbb{N}}$ are tight if, for any $\varepsilon>0$ there exist a compact set $K_\varepsilon\subset E$ such that
$$\mc{L}(X_n)(K_\varepsilon) \ge 1-\varepsilon, \qquad n=1,2,\cdots.$$
\end{definition}
The proof of the following theorem can be seen in \cite[Theorem 2]{jakubowski1998almost}.
\begin{theorem}[Jakubowski version of Skorokhod Theorem]\label{thm:J_Skorokhod}
   Let $(\mc{X},\tau)$ be a topological space with the assumption that there exists a countable family $\{ f_n:\mc{X}\rightarrow \mb{R}\}_{n\in I}$ of $\tau$-continuous functions, which separates points of $\mc{X}$.   
  Assume moreover that the laws of $(X_n)_{n\in I}$ are tight in $\mc{X}$. Then one can find a subsequence $\{X_{n_{k}}\}_{k\in N}$ and $\mc{X}$-valued random variables $\{Y_{k}\}_{k\in N}$ defined on $([0,1],\mc{B}_{[0,1]})$ such that
   $$\mc{L}(X_{n_k})= \mc{L}(Y_k), \qquad k=1,2, \cdots,$$
   $$Y_k(\omega)\rightarrow_{\tau} Y_0(\omega)\;\text{as}\; k\rightarrow \infty, \; \text{a.s. }\omega\in [0,1]$$
   where $\rightarrow_\tau$ represents the convergence of the sequence in the topology $\tau$.
\end{theorem}

\printbibliography
\end{document}